\documentclass[12pt,a4paper]{amsart}
\usepackage[utf8]{inputenc}
\usepackage{amsfonts}
\usepackage{amssymb}
\usepackage{amsmath}
\usepackage{amsthm}
\usepackage{cite}
\usepackage{enumitem}
\usepackage{tikz}
\usepackage{subfigure}
\usepackage{dsfont}

\theoremstyle{plain}
\newtheorem{thm}{Theorem}

\newtheorem{lem}[thm]{Lemma}
\newtheorem{prop}[thm]{Proposition}
\newtheorem{cor}[thm]{Corollary}
\newtheorem{rmk}[thm]{Remark}

\newcommand{\svec}[2]{
\ensuremath{
\begin{pmatrix}
#1 \\ #2 \\
\end{pmatrix}}}

\providecommand{\sm}{\setminus}
\providecommand{\N}{\mathbb{N}}
\providecommand{\R}{\mathbb{R}}
\providecommand{\Z}{\mathbb{Z}}
\providecommand{\C}{\mathbb{C}}

\newcommand{\norm}[1]{\left\lVert #1 \right\rVert}

\renewcommand{\i}{\mathrm{i}}
\renewcommand{\L}[1]{L_\mathrm{rad}^{#1}(\R^3)}

\newcommand{\step}[2]{\vspace*{0.2cm} \underline{Step #1}: \textit{#2} \vspace*{0.1cm}}

\makeatletter
\newcommand{\pushright}[1]{\ifmeasuring@#1\else\omit\hfill$\displaystyle#1$\fi\ignorespaces}
\makeatother

\newcommand{\absatz}{\vspace*{0.5cm}}

\setitemize{itemsep=+2pt}
\setenumerate{itemsep=+2pt}
\begin{document}

\allowdisplaybreaks

\title{Bifurcations of nontrivial solutions of a cubic Helmholtz system}

\author{Rainer Mandel, Dominic Scheider}
\address{R. Mandel, D. Scheider\hfill\break
Karlsruhe Institute of Technology \hfill\break
Institute for Analysis \hfill\break
Englerstra{\ss}e 2 \hfill\break
D-76131 Karlsruhe, Germany}
\email{rainer.mandel@kit.edu}
\email{dominic.scheider@kit.edu}
\date{\today}

\subjclass[2010]{Primary: 35J05, Secondary: 35B32}
\keywords{Nonlinear Helmholtz sytem, bifurcation}

\begin{abstract}
This paper presents local and global bifurcation results for radially symmetric solutions of the cubic Helmholtz system
\begin{equation*}
	\begin{cases}
	-\Delta u - \mu u = \left( u^2 + b \: v^2  \right) u		 &\text{ on } \R^3, 
		\\
	-\Delta v - \nu v = \left( v^2 + b \: u^2  \right) v			 &\text{ on } \R^3.
	\end{cases}
\end{equation*}
It is shown that every point along any given branch of radial semitrivial solutions $(u_0,
0, b)$ or diagonal solutions $(u_b, u_b, b)$ (for $\mu = \nu$) is a bifurcation point.
Our analysis is based on a detailed investigation of the oscillatory behavior of solutions at infinity
that are shown to decay like $\frac{1}{|x|}$ as $|x|\to\infty$. 
\end{abstract}

\maketitle

\allowdisplaybreaks

\section{Introduction and main results}

Systems of two coupled nonlinear Helmholtz equations arise, for instance, 
in models of nonlinear optics. In this paper, we analyze the physically relevant and technically easiest
case of a Kerr-type nonlinearity in $N = 3$ space dimensions, that is, we study the system
\begin{equation}\label{eq_system}\tag{H}
	\begin{cases}
	-\Delta u - \mu u = \left( u^2 + b \: v^2  \right) u	&\text{ on } \R^3, 
		\\
	-\Delta v - \nu v = \left( v^2 + b \: u^2  \right) v	&\text{ on } \R^3
	\end{cases}
\end{equation}
for given $\mu, \nu > 0$ and a constant coupling parameter $b \in \R$. 
We are mostly interested in existence results for fully nontrivial
radially symmetric solutions of this system that we will obtain using bifurcation theory. Such an approach is new in the context of nonlinear Helmholtz equations or systems.
In order to describe the methods used in earlier related works we briefly discuss the available results
for scalar nonlinear Helmholtz equations of the form
\begin{align}\label{eq_helmholtz}
	- \Delta u - \lambda u = Q(x) |u|^{p-2} u \qquad \text{on } \R^N.
\end{align}
Here, the main difficulty is that solutions typically
oscillate and do not belong to $H^1(\R^N)$. In the past years, Ev\'{e}quoz and Weth developed several
methods allowing to find nontrivial solutions of such equations under certain conditions on $Q$ and $p$,
some of which we wish to mention. In~\cite{EvequozWeth_real, EvequozWeth_branch}, they discuss the case of
compactly supported $Q$ and $2 < p < 2^\ast:=\frac{2N}{N-2}$. The idea in~\cite{EvequozWeth_real} is to solve
an exterior problem where the nonlinearity vanishes and knowledge about the far-field expansion of solutions is
available. The remaining problem on a bounded domain can be solved using variational techniques. In
particular, Ev\'{e}quoz and Weth provide a result about the existence and asymptotic behavior of solutions in
the radially symmetric setting, see Theorem~5.2~of~\cite{EvequozWeth_real}.
The approach in~\cite{EvequozWeth_branch} uses Leray-Schauder continuation with respect to the parameter
$\lambda$ in order to find solutions of~\eqref{eq_helmholtz}.
Existence of solutions under the assumption that $Q \in L^\infty(\R^N)$ decays as $|x| \to \infty$ or is
periodic is proved  in~\cite{EvequozWeth} using a dual variational approach, which yields (dual) ground
state solutions and, in the case of decaying $Q$, infinitely many bound states. The technique relies on the
Limiting Absorption Principle of Guti\'{e}rrez, see the explanations before Theorem~6 in~\cite{Gutierrez},
which leads to the additional constraint $\frac{2(N+1)}{N-1} < p < 2^\ast$. Furthermore, assuming that $Q$ is
radial, the existence of a continuum of radially symmetric solutions of~\eqref{eq_helmholtz} has been shown
by Montefusco, Pellacci and the first author in~\cite{radial}. These results rely on ODE techniques and 
only require $p > 2$ and a monotonicity assumption on $Q$.

\absatz

To our knowledge, the only available result on nonlinear Helmholtz systems like~\eqref{eq_system} has been
provided by the authors in~\cite{own} where, using the methods developed in~\cite{EvequozWeth}, 
the existence of a nontrivial dual ground state solution  is proved for the system
\begin{align*}
	&\begin{cases}
		- \Delta u - \mu u = 
		a(x) \left( |u|^\frac{p}{2} + b(x) |v|^\frac{p}{2}  \right)
		|u|^{\frac{p}{2} - 2}u 						& \text{on } \R^N,
		\\
		- \Delta v - \nu v \, = a(x) \left( |v|^\frac{p}{2} + b(x) |u|^\frac{p}{2}  \right)
		|v|^{\frac{p}{2} - 2}v					  		& \text{on } \R^N,
		\\
		u, v \in L^p(\R^N)
	\end{cases}
\end{align*}
for $N \geq 2$, certain nonnegative, $\Z^N$-periodic coefficients $a, b \in L^\infty(\R^N)$ with $0\leq b(x)\leq p-1$ and
$\frac{2(N+1)}{N-1} < p < 2^\ast$. Under additional easily verifiable
assumptions  the ground state can be shown to be fully nontrivial, i.e., both components are nontrivial.
Given the above assumptions this result does not apply in the case of a cubic nonlinearity $p=4$
when $N=3$ which we dicuss in the present paper. In contrast to~\cite{own} we construct fully nontrivial,
radially symmetric solutions for arbitrarily large and small $b\in\R$ that, however, need not be dual ground states.  

\absatz

Our results are inspired by known bifurcation results for the nonlinear Schrödinger system 
\begin{align}\label{eq_systemSchroed}
	\begin{cases}
		- \Delta u + \lambda_1 u = \mu_1 u^3 + b \: u v^2		& \text{on } \R^N,
		\\
		- \Delta v + \lambda_2 v = \mu_2 v^3 + b \: v u^2			 & \text{on } \R^N,
		\\
		u, v \in H^1(\R^N), \quad u > 0, v > 0
	\end{cases}
\end{align}
where one assumes $\lambda_1,\lambda_2>0$ in contrast to \eqref{eq_system}. We focus on bifurcation results by Bartsch,
Wang and Wei in~\cite{Bartsch1} and Bartsch, Dancer and Wang in~\cite{Bartsch2} and refer to the respective introductory sections for a general overview of methods and results for~\eqref{eq_systemSchroed}. In the first-mentioned paper positive solutions
of \eqref{eq_systemSchroed} are found in Theorem~1.1 where the authors show that a continuum consisting of
positive radially symmetric solutions $(u, v, \lambda_1, \lambda_2, \mu_1, \mu_2, b)$ with topological dimension at least $5$ bifurcates
from a two-dimensional set of semipositive solutions $(u, v) = (u_{\lambda_1, \mu_1}, 0)$ parametrized by
$\lambda_1,\mu_1>0$. The existence of countably many bifurcation points giving rise to sign-changing radially
symmetric  solutions was proved by the first author in his dissertation thesis (Satz~2.1.6 of~\cite{DissMandel}). 
 
In the special case $N=2,3$ and  $\lambda_1 = \lambda_2>0$ and $\mu_1, \mu_2 > 0$, Bartsch, Dancer and Wang proved in~\cite{Bartsch2} the existence of countably many mutually disjoint global continua of solutions bifurcating from some diagonal solution family of the form
\begin{align*}
	\{(u_b, u_b, b) : b > -1 \}\subset 
	H^1_\text{rad}(\R^N) \times H^1_\text{rad}(\R^N) \times \R
\end{align*}
with a concentration of bifurcation points as $b \searrow -1$. Here $u_b := (1+b)^{-1/2}u_0$
where $u_0\in H^1_\text{rad}(\R^N)$ is a nondegenerate solution of the nonlinear Schrödinger equation
$-\Delta u + u = u^3$. Moreover, having introduced a suitable labeling of the continua, the authors showed
that the $k$-th continuum consists of solutions where the radial profile of $u-v$ has exactly $k-1$ nodes,
cf. Theorem~2.3 in~\cite{Bartsch2}.

\absatz

In the Helmholtz case, we will analyze the corresponding cases of bifurcations from semitrivial and
diagonal solutions in Theorems~\ref{THEtheorem}~and~\ref{THEtheorem_diagonal}, respectively. In contrast to the
Schrödinger case, we will show that bifurcation occurs at every point in a topology suitable for the system~\eqref{eq_system}. Looking more closely, we find the same structure of discrete bifurcation points
as in the Schrödinger case when fixing a set of asymptotic parameters $\tau_1, \omega$ prescribing the
oscillatory behavior of solutions as $|x| \to \infty$ as in the
conditions~\eqref{eq_asymptotic}~respectively~\eqref{eq_asymptotic_diagonal} below. Whereas, in the
Schrödinger case, the bifurcating solutions are characterized by their nodal structure, 
we characterize them in the Helmholtz case by a condition on the ''asymptotic phase'' of the solution
(disguised as an integral), which at least close to the $j$-th bifurcation point takes the
value $\omega + j \pi$, see Proposition~\ref{prop_pruefer} and the explanations following it.

\absatz

We now present our main results.
Motivated by the decay properties of radially symmetric solutions of nonlinear Helmholtz equations in~\cite{radial}, e.g. Theorem~1.2~(iii), we look for solutions in the Banach space $X_1$ where, for $q \geq 1$,
\begin{align}\label{eq_xq}
	X_q := \left\{ w \in C_\text{rad}(\R^3, \R) \: | \: \norm{w}_{X_q} < \infty \right\}
	\quad\text{where}\quad 
	\norm{w}_{X_q} := \sup_{x \in \R^3} (1 + |x|^2)^\frac{q}{2} |w(x)|.
\end{align}
Working on these spaces, we will be able to derive compactness properties which are crucial when proving our bifurcation results.
Throughout, we discuss classical, radially symmetric solutions $u, v \in X_1 \cap C^2(\R^3)$ of the system~\eqref{eq_system} and related equations. Let us remark here only briefly that, using elliptic regularity, all weak solutions in $u, v \in L^4_\text{rad}(\R^3)$ are actually smooth and, thanks to Proposition~\ref{prop_pruefer} in the next section, belong to $X_1 \cap C^2(\R^3)$. 

In our first result, we study bifurcation of solutions $(u, v, b)$ of the nonlinear Helmholtz system~\eqref{eq_system} 
from a branch of semitrivial solutions of the form 
\begin{align*}
	\mathcal{T}_{u_0} := \{ (u_0, 0, b) \: | \: b \in \R \} \subseteq X_1 \times X_1 \times \R
\end{align*}
in the Banach space $X_1 \times X_1 \times \R$. In contrast to the Schrödinger case, we will demonstrate
that for each of the uncountably many radial solutions $u_0\in X_1\cap C^2(\R^3)$ (see~\cite{radial}) of the
scalar problem
\begin{align}\label{eq_singleHH}\tag{h}
	- \Delta u_0 - \mu u_0 = u_0^3
	\qquad \text{on } \R^3
\end{align}
we have that every point in $\mathcal{T}_{u_0}$ is a bifurcation point for fully nontrivial solutions
of~\eqref{eq_system}. As a consequence of some auxiliary results in Section~\ref{sect_scalar}, we will derive
the following result on the scalar Helmholtz equation at the end of Section~\ref{sect_prfs}.

\begin{prop}\label{prop_u0-sigma0-tau0}
Let $\mu > 0$ and let $u_0 \in X_1\cap C^2(\R^3)$ be any radially symmetric solution  of the nonlinear
Helmholtz equation~\eqref{eq_singleHH}. Then $u_0$ satisfies
\begin{align*}
	u_0(x) = c_0 \: \frac{\sin(|x| \sqrt{\mu} + \sigma_0)}{|x|} + O\left(\frac{1}{|x|^2}\right) 
	\qquad \mathrm{as } \: |x| \to \infty
\end{align*}
for some constants $c_0 \neq 0$ and $\sigma_0 \in [0, \pi)$, and there exists a unique $\tau_0 \in [0, \pi)$ with the property that the problem
\begin{align*}
	\begin{cases}
	- \Delta w - \mu w = 3 u_0^2(x) \: w  &\mathrm{on } \: \R^3,
	\\
	w(x) = \frac{\sin(|x| \sqrt{\mu} + \tau_0)}{|x|} + O\left(\frac{1}{|x|^2}\right) & \mathrm{as } \: |x| \to \infty
	\end{cases}
\end{align*}
admits a nontrivial solution $w_0 \in X_1 \cap C^2(\R^3)$. Moreover, this solution $w_0$ is unique.
\end{prop}

Here and in the following, we fix $\mu, \nu > 0$ and $u_0 \in X_1\cap C^2(\R^3)$ with associated constants $\sigma_0, \tau_0 \in [0, \pi)$ as in Proposition~\ref{prop_u0-sigma0-tau0}. For $\tau_1 \in [0, \pi) \setminus \{ \tau_0 \}$, we observe in particular that
\begin{equation}\label{eq_nondegen-tau0}\tag{N}
	\begin{cases}
	- \Delta w - \mu w = 3 u_0^2(x) \: w  &\mathrm{on } \: \R^3,
	\\
	w(x) = \frac{\sin(|x| \sqrt{\mu} + \tau_1)}{|x|} + O\left(\frac{1}{|x|^2}\right) & \mathrm{as } \: |x| \to \infty
	\end{cases}
	\quad \text{implies} \quad
	w = 0.
\end{equation}
This nondegeneracy property will be used later to prove that the linearization of the system~\eqref{eq_system} close to points $(u_0, 0, b_0) \in \mathcal{T}_{u_0}$ admits at most one-dimensional kernels. 
Our strategy will be to use bifurcation from simple eigenvalues with $b$ acting as a bifurcation parameter. The existence of isolated and algebraically simple eigenvalues will be ensured by assuming radial symmetry and by imposing suitable conditions on the asymptotic behavior of the solutions $u, v$. 
For $\tau_1, \omega \in [0, \pi)$ with $\tau_1 \neq \tau_0$, we define $\mathcal{S}(\omega) \subseteq X_1 \times X_1 \times \R \setminus \mathcal{T}_{u_0}$ as the set of all solutions $(u, v, b) \in X_1 \times X_1 \times \R \setminus \mathcal{T}_{u_0}$ of~\eqref{eq_system} satisfying the asymptotic conditions
\begin{equation}\label{eq_asymptotic}\tag{$A_{\omega}$}
	\begin{split}
	&u(x) - u_0(x) = c_u \: \frac{\sin(|x| \sqrt{\mu} + \tau_1)}{|x|} + O\left(\frac{1}{|x|^2}\right)
	\\
	&v(x) = c_v \: \frac{\sin(|x| \sqrt{\nu} + \omega)}{|x|} + O\left(\frac{1}{|x|^2}\right)
	\end{split}
	\qquad
	\text{as } |x| \to \infty
\end{equation}
for some $c_u, c_v \in \R$. Propositions~\ref{prop_cont-comp}~and~\ref{prop_pruefer} in the following section will show that an asymptotic behavior of such form is natural to assume for solutions of the system~\eqref{eq_system}.
We emphasize that we do not denote the dependence of the set $\mathcal{S}(\omega)$ and of the asymptotic conditions~\eqref{eq_asymptotic} on the choice $\tau_1 \in [0, \pi) \setminus \{ \tau_0 \}$.
With that, we obtain the following

\begin{thm}\label{THEtheorem}
Let $\mu, \nu > 0$, fix any $u_0 \in X_1$ solving the nonlinear Helmholtz equation~\eqref{eq_singleHH} and choose $\tau_1 \in [0, \pi) \setminus \{ \tau_0 \}$ with $\tau_0$ as in Proposition~\ref{prop_u0-sigma0-tau0}. 
Then, for every $\omega\in [0, \pi)$, there exists a strictly increasing sequence $(b_k(\omega))_{k \in \Z}$ such that $(u_0, 0, b_k(\omega)) \in \overline{\mathcal{S}(\omega)}$ where $\mathcal{S}(\omega)$ denotes the set of all solutions $(u, v, b) \in X_1 \times X_1 \times \R \setminus \mathcal{T}_{u_0}$ of~\eqref{eq_system} satisfying~\eqref{eq_asymptotic}.  Moreover,
\begin{itemize}
\item[(i)] the respective connected components $\mathcal{C}_k(\omega)$ of $(u_0, 0, b_k(\omega))$ in $\overline{\mathcal{S}(\omega)}$ are unbounded in $X_1 \times X_1 \times \R$; and
\item[(ii)] each bifurcation point $(u_0, 0, b_k(\omega))$ has a neighborhood where the set $\mathcal{C}_k(\omega)$ is a smooth curve in $X_1 \times X_1 \times \R$ which, except for the bifurcation point, consists of fully nontrivial solutions.
\end{itemize}
\end{thm}

The main tools in proving this statement are the Crandall-Rabinowitz Bifurcation
Theorem, which will be used to show the local statement (ii) of Theorem~\ref{THEtheorem}, and Rabinowitz' Global Bifurcation Theorem, which will provide (i). For a reference, see \cite{CrandRab},~Theorem~1.7 and \cite{Rab},~Theorem~1.3.
We add some remarks the proof of which will be given after having proved Theorem~\ref{THEtheorem} in
Section~\ref{sect_prf1}.

\begin{rmk}\label{THEremark}
\begin{itemize}
\item[(a)]
We will also see that fully nontrivial solutions of~\eqref{eq_system} satisfying the asymptotic
condition~\eqref{eq_asymptotic} bifurcate from some point $(u_0, 0, b) \in \mathcal{T}_{u_0}$ if and only if $b =
b_k(\omega)$ for some $k \in \Z$.
\item[(b)]
Furthermore, we will prove that the map 
$\R \to \R, k \pi + \omega \mapsto b_k(\omega)$ where $0 \leq \omega < \pi, k \in \Z$
is strictly increasing and onto with $b_k(\omega) \to \pm \infty$ as $k \to \pm \infty$.  
Thus, in particular, every point $(u_0, 0, b) \in \mathcal{T}_{u_0}$, $b \in \R$, is a bifurcation point for fully nontrivial radial solutions of~\eqref{eq_system}, which is in contrast to the case of Schrödinger systems where bifurcation points are isolated, cf. \cite{DissMandel},~Satz~2.1.6. 
\item[(c)]
Close to the respective bifurcation point $(u_0, 0, b_k(\omega)) \in \mathcal{T}_{u_0}$, each continuum $\mathcal{C}_k(\omega)$ is characterized by a phase parameter $\omega + k \pi$ derived from the asymptotic behavior of $v$. It seems that, in the
Helmholtz case of oscillating solutions, the integer $k$ takes the role of the nodal characterizations in the
Schrödinger case, cf. Satz~2.1.6~in~\cite{DissMandel}.  That phase parameter is constant on connected
subsets of the continuum until it possibly runs into another family of semitrivial solutions $\mathcal{T}_{u_1}$ with $u_1 \neq u_0$;
unfortunately we cannot provide criteria deciding whether or not this happens. For this reason we cannot claim that $C_k(\omega)$ contains an unbounded sequence of fully nontrivial solutions.
\item[(d)]
The condition $\tau_1 \neq \tau_0$ is a nondegeneracy condition which ensures that the simplicity requirements of the above-mentioned bifurcation theorems are satisfied. 
If we additionally impose $\tau_1 \neq \sigma_0$, we infer $u \neq 0$ for any solution $(u, v, b) \in
\mathcal{C}_k(\omega)$. 
Moreover, the proof will show that the values $b_k(\omega)$ do not depend on the choice of $\tau_1$.
\end{itemize}
\end{rmk}

In our second result we provide a counterpart of the global bifurcation result by Bartsch, Dancer and
Wang~\cite{Bartsch2} described earlier. The authors obtained infinitely (but countably) many mutually
disjoint continua of solutions bifurcating from a diagonal solution family of~\eqref{eq_systemSchroed} with
elements $(u_b, u_b, b)$ where $u_b = (1+b)^{-1/2}u_0$ and $u_0$ is a nondegenerate solution of $-\Delta u +
u = u^3$. Using the same functional analytical setup as in Theorem~\ref{THEtheorem}, we find an analogue of
these results for the nonlinear Helmholtz system~\eqref{eq_system}. For $u_0$ as in
Proposition~\ref{prop_u0-sigma0-tau0} and $\tau_1, \omega \in [0, \pi)$, $\tau_1 \neq \tau_0$, we introduce the diagonal solution family
\begin{align*}
	\mathfrak{T}_{u_0} := \left\{
		(u_b, u_b, b) \: \big| \: b > -1
	\right\} \subseteq X_1 \times X_1 \times \R
	\qquad \text{with } u_b := (1+b)^{-1/2} \: u_0
\end{align*}
and denote by $\mathfrak{S}(\omega)$ the set of all solutions $(u, v, b) \in X_1 \times X_1 \times \R \setminus \mathfrak{T}_{u_0}$ of the nonlinear Helmholtz system~\eqref{eq_system} with 
\begin{align}\label{eq_asymptotic_diagonal}\tag{$A_\omega^\text{diag}$}
	\begin{split}
	u(x) + v(x) &= 2 u_b(x) + \tilde{c} \: \frac{\sin(|x| \sqrt{\mu} + \tau_1)}{|x|} 
	+ O\left(\frac{1}{|x|^2}\right)
	\\
	u(x) - v(x) &= c \: \frac{\sin(|x| \sqrt{\mu} + \omega)}{|x|} + O\left(\frac{1}{|x|^2}\right)
	\end{split}
	\qquad
	\text{as } |x| \to \infty
\end{align}
for some $\tilde{c}, c \in \R$. Our existence result for fully nontrivial solutions of~\eqref{eq_system}
bifurcating from~$\mathfrak T$ with asymptotics \eqref{eq_asymptotic_diagonal} reads as follows.

\begin{thm}\label{THEtheorem_diagonal}
Let $\mu> 0$, fix any $u_0 \in X_1$ solving the nonlinear Helmholtz equation~\eqref{eq_singleHH} and choose $\tau_1 \in [0, \pi) \setminus \{ \tau_0 \}$ with $\tau_0$ as in Proposition~\ref{prop_u0-sigma0-tau0}. 
Then, for every $\omega\in [0, \pi)$, there exists a sequence $(\tilde{b}_k(\omega))_{k \in \Z}$ 
such that $(u_{\tilde{b}_k(\omega)}, u_{\tilde{b}_k(\omega)}, \tilde{b}_k(\omega)) \in \overline{\mathfrak{S}(\omega)}$ where
 $\mathfrak{S}(\omega)$ denotes the set of all solutions $(u, v, b) \in X_1 \times X_1 \times \R \setminus
 \mathfrak{T}_{u_0}$ of~\eqref{eq_system} satisfying~\eqref{eq_asymptotic_diagonal}. Moreover,
\begin{itemize}
\item[(i)] the respective connected components $\mathfrak{C}_k(\omega)$ of $(u_{\tilde{b}_k(\omega)}, u_{\tilde{b}_k(\omega)}, \tilde{b}_k(\omega))$ in $\overline{\mathfrak{S}(\omega)}$ are unbounded in $X_1\times X_1 \times \R$; and
\item[(ii)] each bifurcation point $(u_{\tilde{b}_k(\omega)}, u_{\tilde{b}_k(\omega)}, \tilde{b}_k(\omega))$ has a neighborhood where the set
$\mathfrak{C}_k(\omega)$ contains a smooth curve in $X_1 \times X_1 \times \R$ which, except for the bifurcation point, consists of fully nontrivial, non-diagonal solutions.
\end{itemize}
\end{thm}

Again, similar statements as in Remark~\ref{THEremark} can be proved. In particular, one can check that every
point on $\mathfrak T_{u_0}$ is a bifurcating point by a suitable choice of $\omega$.

\absatz

Let us give a short outline of this paper. 
In the Section~2, we introduce the concepts and technical results we use in the proof of 
Theorems~\ref{THEtheorem}~and~\ref{THEtheorem_diagonal}, which are presented in the
Section~\ref{sect_prf1}~and Section~\ref{sect_prf2}. In the final section, we provide the proofs of the
auxiliary results of Section~\ref{sect_scalar}. Among those, we also prove
Proposition~\ref{prop_u0-sigma0-tau0} from above.

\section{On the scalar problem. Spectral properties}\label{sect_scalar}

The main challenge in proving Theorem~\ref{THEtheorem} is a thorough analysis of the linearized problem which we provide in this chapter.
Throughout, we fix $\lambda > 0$ and discuss the linear Helmholtz equation
\begin{align}\label{eq_single}
	- \Delta w - \lambda w = f 
	\qquad \text{on } \R^3
\end{align}
for some $f \in X_3$, where $X_3$ is defined in~\eqref{eq_xq}. We will frequently identify radially symmetric functions $x \mapsto w(x)$ with their profiles; in particular, we denote by $w' := \partial_r w, w'' = \partial_r^2 w$ the radial derivatives.
The results we establish in this section will demonstrate how to rewrite the system~\eqref{eq_system} in a way suitable for Bifurcation Theory.  

\subsection{Representation Formulas}

First, we discuss a representation formula for solutions of the linear inhomogeneous Helmholtz equation~\eqref{eq_single}. 
The results resemble a Representation Theorem by Agmon, Theorem~4.3 in~\cite{agmon}, but in our setting, the proof is much easier. 
To this end, we introduce the fundamental solutions 
\begin{equation}\label{eq_fundamental}
	\Psi_\lambda, \tilde{\Psi}_\lambda: \R^3 \to \R, 
	\quad
	\Psi_\lambda(x) := \frac{\cos(\sqrt{\lambda}|x|)}{4 \pi |x|}
	\quad \text{and} \quad
	\tilde{\Psi}_\lambda(x) := \frac{\sin(\sqrt{\lambda} |x|)}{4 \pi |x|}
	\quad (x \neq 0)
\end{equation}
of the equation $- \Delta w - \lambda w = 0$ on $\R^3$. 
Throughout, we will require knowledge of the mapping properties of convolutions with $\Psi_\lambda$ resp. $\tilde{\Psi}_\lambda$. 
Various results of such type have been established by Ev\'{e}quoz and Weth in~\cite{EvequozWeth} and further publications, assuming $f \in L^{p'}(\R^N)$ and $w \in L^{p}(\R^N)$ for suitable $p, p' \in (1, \infty)$. In the spaces $X_3$ resp. $X_1$, which satisfy the continuous embeddings
\begin{align}\label{eq_embed}
	X_1 \hookrightarrow \L{p}
	\text{ for } 3 < p \leq \infty,
	\qquad
	X_3 \hookrightarrow \L{q}
	\text{ for } 1 < q \leq \infty,
\end{align}
we prove the following stronger statements. In particular, we obtain a compactness result, which will be most useful in order to establish spectral results such as Proposition~\ref{prop_spectrum} below.

\begin{prop}\label{prop_cont-comp}
	For arbitrary constants $\alpha, \tilde{\alpha} \in \R$, we formally introduce the convolution operator 
	$\mathcal{R}_\lambda f := \left( \alpha \Psi_\lambda + \tilde{\alpha} \tilde{\Psi}_\lambda \right) \ast f$. 
	Then, 
	\begin{itemize}
	\item[(a)] 
	the linear map $\L{\frac{4}{3}} \to \L{4}, \: \: f \mapsto \mathcal{R}_\lambda f$ is well-defined and continuous;
	\item[(b)]
	the linear map $X_3 \to X_1, \: \: f \mapsto \mathcal{R}_\lambda f$ is well-defined, continuous and compact;
	\item[(c)] for $f \in X_3$, we have $w := \mathcal{R}_\lambda f \in X_1 \cap C^2(\R^3)$ 
	with $- \Delta w - \lambda w = \alpha \cdot f$ on $\R^3$; 	and
	\item[(d)] for $f \in X_3$, the profile of $w := \mathcal{R}_\lambda f$ satisfies the asymptotic identity
	\begin{align*}
		w(r) = \sqrt{\frac{\pi}{2}} \: \hat{f}(\sqrt{\lambda}) \cdot \frac{\alpha \cos(r \sqrt{\lambda}) + \tilde{\alpha} \sin(r \sqrt{\lambda})}{r} 
		+ \frac{\delta_f(r)}{r^2} \cdot (|\alpha| + |\tilde{\alpha}|)
		\quad
		\text{as } r \to \infty
	\end{align*}
	where $| \delta_f(r) | \leq \frac{2}{\sqrt{\lambda}} \cdot \norm{f}_{X_3}$ as well as
	$\hat{f}(\sqrt{\lambda}) = \sqrt{\frac{2}{\pi}} \int\limits_0^\infty f(r) \frac{\sin(r\sqrt{\lambda})}{r\sqrt{\lambda}} \: r^2 \: \mathrm{d}r$.
	Further, $\tilde{\Psi}_\lambda \ast f = 4 \pi \sqrt{\frac{\pi}{2}} \: \hat{f}(\sqrt{\lambda}) \cdot \tilde{\Psi}_\lambda$, and
	the radial derivative satisfies
	\begin{align*}
		w'(r) = \sqrt{\frac{\pi}{2}} \: \hat{f}(\sqrt{\lambda}) \cdot \frac{- \alpha \sqrt{\lambda} \sin(r \sqrt{\lambda}) + \tilde{\alpha} \sqrt{\lambda} \cos(r \sqrt{\lambda})}{r} 
		+ O \left(\frac{1}{r^2}\right)
		\quad
		\text{as } r \to \infty.
	\end{align*}
	\end{itemize}
\end{prop}
This was motivated by yet unpublished results provided by Ev\'{e}quoz, which in case $N = 3$ yield a constant $C(\lambda) > 0$ with
\begin{align*}
	\norm{\min\{|\,\cdot\,|,|\,\cdot\,|^{\frac{3}{2}}\} 
	\cdot \left| (\Psi_\lambda + \i \tilde{\Psi}_\lambda) \ast f \right| \: }_{L^\infty(\R^3)} \leq 
	C(\lambda) \cdot \norm{f}_{L^\frac{4}{3}(\R^3)}
	\quad \text{for all } f \in \mathcal{S}_\text{rad}(\R^3).
\end{align*}
The fact that we choose the stronger topology of $X_3$ instead of $\L{\frac{4}{3}}$ will imply that the convolution even maps to $L^\infty_\text{rad}(\R^3)$ without additional weight at the origin and provide further compactness properties. 
Moreover, the proof will show that the decay rate prescribed by the $X_3$ space is not the optimal one yielding continuity and compactness as in (b). In fact, the same proof yields continuity even if $X_3$ is replaced by $X_{2 + \varepsilon}$ ($\varepsilon > 0$) and compactness assuming $X_{9/4 + \varepsilon}$ ($\varepsilon > 0$).

From Proposition~\ref{prop_cont-comp}, we now derive the representation formulae we require later to construct  the functional analytic setting in the proof of Theorem~\ref{THEtheorem}.
For $\omega \in (0, \pi)$, we define the linear convolution operators
\begin{equation}\label{eq_R-lambda-omega}
		\mathcal{R}_\lambda^\omega: X_3 \to X_1, \quad
		f \mapsto \Psi_\lambda \ast f + \cot(\omega) \: \tilde{\Psi}_\lambda \ast f
\end{equation}
which provide solutions of the Helmholtz equation~\eqref{eq_single} the asymptotic behavior of which is described by the phase parameter $\omega$ as follows. 
\begin{cor}\label{cor_asymptotic}
	Let $\omega \in (0, \pi)$ and $f \in X_3$.
	Then, for $w \in X_1$, we have $w = \mathcal{R}_\lambda^\omega f$
	if and only if $w$ is a $C^2$ solution of $- \Delta w - \lambda w = f$ on $\R^3$
	with asymptotic behavior 
	\begin{align*}
		w(x) = \gamma \cdot \frac{\sin(|x| \sqrt{\lambda} + \omega)}{|x|} 
		+ O\left(\frac{1}{|x|^2}\right)
		\qquad
		\text{as } |x| \to \infty
	\end{align*}
	for some $\gamma \in \R$.
\end{cor}

We observe that the operator $\mathcal{R}_\lambda^\omega$ is not well-defined for $\omega = 0$ due to the pole of the cotangent. Thus we have to extend some of the previous results in a suitable framework. First, by the Hahn-Banach Theorem, we construct continuous linear functionals $\alpha, \beta \in X_1'$ as follows. On the linear subspace 
\begin{align*}
	U_1(\lambda) := \bigg\{
		w \in X_1 \: \bigg| \:
		&w(x) = \alpha_w \frac{\sin(|x|\sqrt{\lambda})}{4 \pi|x|} 
		+ \beta_w \frac{\cos(|x|\sqrt{\lambda})}{4 \pi|x|} + 
		O\left(\frac{1}{|x|^2}\right) 
		\\
		&\text{as } |x| \to \infty \text{ for some } \alpha_w, \beta_w \in \R
	\bigg\},
\end{align*}
we let, for $w \in U_1(\lambda)$ 
with $w(r) = \alpha_w \frac{\sin(r\sqrt{\lambda})}{4 \pi r} 	+ \beta_w \frac{\cos(r \sqrt{\lambda})}{4 \pi r} + 
O\left(\frac{1}{r^2}\right)$ as $r = |x| \to \infty$,
\begin{equation}\label{eq_alpha-beta}
\begin{split}
	\alpha^{(\lambda)}(w) &:= \alpha_w = \lim_{n \to \infty} \left[
	4 \pi \cdot \frac{2\pi n + \frac{\pi}{2}}{\sqrt{\lambda}} \cdot w\left( \frac{2\pi n + \frac{\pi}{2}}{\sqrt{\lambda}} \right)
	\right], 
	\\ 
	\beta^{(\lambda)}(w) &:= \beta_w = \lim_{n \to \infty} \left[
	4 \pi \cdot \frac{2 \pi n}{\sqrt{\lambda}} \cdot w\left( \frac{2\pi n}{\sqrt{\lambda}} \right)
	\right]. 
\end{split}
\end{equation}
But then $|\alpha^{(\lambda)}(w)|, |\beta^{(\lambda)}(w)| \leq \limsup_{r \to \infty} |4\pi \sqrt{1 + r^2} \cdot w(r)| \leq 4 \pi \norm{w}_{X_1}$ for $w \in U_1(\sqrt{\lambda})$; hence, after continuous extension, $\alpha^{(\lambda)}, \beta^{(\lambda)} \in X_1'$. In particular, for any $f \in X_3$ and $\lambda > 0$, Proposition~\ref{prop_cont-comp}~(d) implies $\Psi_\lambda \ast f, \tilde{\Psi}_\lambda \ast f \in U_1(\lambda)$ with
\begin{equation}\label{eq_alpha-beta-coeff}
\begin{split}
	&\alpha^{(\lambda)}(\Psi_\lambda \ast f) = \beta^{(\lambda)}(\tilde{\Psi}_\lambda \ast f)  = 0, 
	\\
	&\alpha^{(\lambda)}(\tilde{\Psi}_\lambda \ast f) 
	= \beta^{(\lambda)}(\Psi_\lambda \ast f) = 4\pi \sqrt{\frac{\pi}{2}} \cdot \hat{f}(\sqrt{\lambda}).
\end{split}
\end{equation}

We find characterizations as in Corollary~\ref{cor_asymptotic} both without any asymptotic condition and in the case $\omega = 0$:
\begin{cor}\label{cor_alpha-beta}
	Let $f \in X_3$ and $w \in X_1$, $\omega \in [0, \pi)$, and consider continuous linear functionals
	$\alpha^{(\lambda)}, \beta^{(\lambda)} \in X_1'$ satisfying~\eqref{eq_alpha-beta}.
	Then the following characterizations hold:
	\begin{itemize}
	\item[(a)]
	$w$ is twice continuously differentiable and solves $- \Delta w - \lambda w = f$ on $\R^3$ 
	if and only if $w = \Psi_\lambda \ast f + \alpha^{(\lambda)}(w) \cdot \tilde{\Psi}_\lambda$.
	\item[(b)]
	Let $\sigma \in \{ -1, +1 \}$.
	$w$ is twice continuously differentiable, solves $- \Delta w - \lambda w = f$ on $\R^3$ and satisfies
	\begin{align*}
		w(x) = \gamma \: \frac{\sin(|x|\sqrt{\lambda})}{|x|} + O\left(\frac{1}{|x|^2}\right)
		\qquad \text{as } |x| \to \infty
	\end{align*}
	for some $\gamma \in \R$
	if and only if $w = \Psi_\lambda \ast f + (\alpha^{(\lambda)}(w) + \sigma \beta^{(\lambda)}(w)) \cdot \tilde{\Psi}_\lambda$. 
	In this case, $\beta^{(\lambda)}(w) = 0$.
	\end{itemize}
\end{cor}

\subsection{The Asymptotic Phase}

Frequently, equations of interest will take the form~\eqref{eq_single} with $f = g \cdot w$ for some $g \in X_2$, see~\eqref{eq_xq}. 
We can then use ODE methods, more specifically the Prüfer transformation, to discuss the corresponding initial value problem for the profiles,
\begin{align}\label{eq_single-pruefer}
	- w'' - \frac{2}{r} w' - \lambda w = g(r) \: w
	\qquad \text{on } (0, \infty)
	\qquad \text{with } w(0) = 1, \: w'(0) = 0.
\end{align}

\begin{prop}\label{prop_pruefer}
	Assume $g \in X_2$. Then the ODE initial value problem~\eqref{eq_single-pruefer}
	has a unique (global) solution $w: [0, \infty) \to \R$ which asymptotically satisfies
	\begin{align*}
		w(r) &= \rho_\lambda(g) \: \frac{\sin(r \sqrt{\lambda} + \omega_\lambda(g))}{r} + O\left(\frac{1}{r^2}\right), 
		\\
		w'(r) &= \rho_\lambda(g) \sqrt{\lambda} \: \frac{\cos(r \sqrt{\lambda}  + \omega_\lambda(g))}{r} + O\left(\frac{1}{r^2}\right)
	\end{align*}
	as $r \to \infty$ for some $\rho_\lambda(g) > 0$ and $\omega_\lambda(g) \in \R$.
	Here, the value of $\omega_\lambda(g)$ is given by 
	\begin{equation}\label{eq_aspt-phase}
	\begin{split}
		&\omega_\lambda(g) = \frac{1}{\sqrt{\lambda}} \int_0^\infty g(r) \sin^2(\phi(r)\sqrt{\lambda}) \, \mathrm{d}r
		\\
		&\text{where } \phi: [0, \infty) \to \R \text{ solves }
		\begin{cases}
		\phi' = 1 + \frac{1}{\lambda} g(r) \sin^2(\phi \sqrt{\lambda}),
		\\
		\phi(0) = 0.
		\end{cases}
	\end{split}
	\end{equation}
\end{prop} 

We will refer to the term $\omega_\lambda(g)$ as the asymptotic phase of the solution $w$ of~\eqref{eq_single-pruefer}; we suggest to think of it as a way of quantifying the effect of the right-hand side of equation~\eqref{eq_single-pruefer} on the solution $w$ in a situation where solutions typically oscillate. More precisely, writing $\omega_\lambda(g) = \omega + k \pi$ for some $k \in \Z$ and $\omega \in [0, \pi)$, the parameter $\omega$ describes the shift of phase between the profile $r \cdot w(r)$ and $\sin(r \sqrt{\lambda})$ at large radii; and the profile $r \cdot w(r)$ attains $k$ additional nodes when compared with $\sin(r \sqrt{\lambda})$ in sufficiently large intervals containing $0$.

Looking back to the asymptotic conditions imposed in Corollaries~\ref{cor_asymptotic}~and~\ref{cor_alpha-beta}, we see that they are of the form 
\begin{align*}
	- \Delta w - \lambda w = g \cdot w \quad \text{on } \R^3,
	\qquad
	\omega_\lambda(g) \in \omega + \pi \Z.
\end{align*}
Such boundary conditions at infinity will provide operators with spectral properties suitable for building the functional analytic framework in which to prove Theorem~\ref{THEtheorem}.  

\begin{rmk}
	The previous results are closely related to those in Corollary~\ref{cor_asymptotic}. 
	In fact, comparing the asymptotic expansions in Corollary~\ref{cor_asymptotic} applied with $f = g \cdot w$ 
	and Proposition~\ref{prop_pruefer}, we identify $\omega_\lambda(g) \in \omega + \pi \Z$ and $\rho_\lambda(g) = |\gamma|$.
	
	We point out two aspects in which Proposition~\ref{prop_pruefer} provides stronger statements: 
	First, there is no singularity in case $\omega = 0$ as it appears in the
	definition~\eqref{eq_R-lambda-omega} of the convolution operators $\mathcal{R}_\lambda^\omega$. 
	Second, we explicitly have $\rho_\lambda(g) > 0$.
	However, in order to construct the functional analytic setting when proving Theorem~\ref{THEtheorem},
	we will use the convolution operators $\mathcal{R}_\lambda^\omega$ due to their
	differentiability and compactness properties, see Proposition~\ref{prop_cont-comp}. The ODE results 
	will then be helpful to extract spectral properties.
\end{rmk}

As a first auxiliary result, we prove the following continuity property.

\begin{prop}\label{prop_aspt-phase-continuous}
The asymptotic phase is continuous as a map $\omega_\lambda: X_2 \to \R, g \mapsto \omega_\lambda(g)$.
\end{prop}

When studying eigenvalue problems of a linearization of~\eqref{eq_system} as often required in Bifurcation Theory, it will be helpful to know the dependence of the asymptotic phase $\omega_\lambda(b \, u_0^2)$ on the (eigenvalue) parameter $b \in \R$. Here we denote by $u_0 \in X_1 \cap C^2(\R^3)$ some solution of $-\Delta u_0 - \mu u_0 = u_0^3$ on $\R^3$. 

\begin{prop}\label{prop_asymptoticphase}
	The map $\R \to \R, \: b \mapsto \omega_\lambda(b \: u_0^2)$ 
	is continuous, strictly increasing and onto with $\omega_\lambda(0) = 0$.
\end{prop}

\subsection{The spectrum of the linearization}

In the proof of Theorem~\ref{THEtheorem}, we will rewrite the nonlinear Helmholtz system~\eqref{eq_system} in the form
\begin{align*}
	u = \mathcal{R}_\mu^\tau (u(u^2 + b \, v^2)),
	\qquad
	v = \mathcal{R}_\nu^\omega (v(v^2 + b \, u^2)),
	\qquad
	u, v \in X_1
\end{align*}
for some $\tau, \omega \in (0, \pi)$, which additionally imposes a certain asymptotic behavior on the solutions, see Corollary~\ref{cor_asymptotic}. In order to analyze the linearized problem, we fix some nontrivial $u_0 \in X_1 \cap C^2(\R^3)$ with $-\Delta u_0 - \mu u_0 = u_0^3$ on $\R^3$ and study the spectra of the linear operators
\begin{align}\label{eq_R-lin}
	\mathbf{R}_\lambda^\omega: X_1 \to X_1, 
	\qquad
	w \mapsto 
	\mathcal{R}_\lambda^\omega (u_0^2 \, w) = \left( \Psi_\lambda + \cot (\omega) \: \tilde{\Psi}_\lambda \right) \ast [u_0^2 \, w],
\end{align}
which are compact thanks to Proposition~\ref{prop_cont-comp}~(b).
We now present the final result in this Section:

\begin{prop}\label{prop_spectrum}
	Let $\omega \in (0, \pi)$, $\lambda > 0$ and $u_0$ as before. 
	By Proposition~\ref{prop_asymptoticphase}, for $k \in \Z$, we define $b_k(\omega, \lambda, u_0^2) \in \R$ via 
	$\omega_\lambda (b_k(\omega, \lambda, u_0^2) \, u_0^2)  = \omega + k \pi$. 
	Then the spectrum of $\mathbf{R}_\lambda^\omega$ is
	\begin{align*}
	\sigma(\mathbf{R}_\lambda^\omega) = \{ 0 \} \cup \sigma_\mathrm{p}(\mathbf{R}_\lambda^\omega),
	\quad
	\sigma_\mathrm{p}(\mathbf{R}_\lambda^\omega) = \left\{ \frac{1}{b_k(\omega, \lambda, u_0^2)} \: \bigg| \: k \in \Z \right\}.
	\end{align*}
	Moreover, all eigenvalues are algebraically simple, and the sequence $(b_k(\omega, \lambda, u_0^2))_{k \in \Z}$ 
	is strictly increasing and unbounded below and above.
\end{prop}
This excludes the case $\omega = 0$, even though the values $b_k(0, \lambda, u_0^2) \in \R$, $k \in \Z$, can be defined accordingly. 
Indeed, the first step of the proof of Proposition~\ref{prop_spectrum} above provides the following statement for all $\omega \in [0, \pi)$: 
\begin{rmk}
Fix $\omega \in [0, \pi)$. Then the problem
\begin{align*}
	- \Delta w - \lambda w = b u_0^2 \: w \quad \text{on } \R^3,
	\qquad
	w(x) = \gamma \frac{\sin(|x|\sqrt{\lambda} + \omega)}{|x|} + O\left(\frac{1}{|x|^2}\right)
	\quad \text{as } |x| \to \infty
\end{align*}
for some $\gamma \in \R$ has a nontrivial radial solution $w \in X_1 \cap C^2(\R^3)$ if and only if $b = b_k(\omega, \lambda, u_0^2)$ for some $k \in \Z$.
\end{rmk}

\section{Proof of Theorem~\ref{THEtheorem}}\label{sect_prf1}

We will first present the proof in case of asymptotic parameters $0 < \omega, \tau_1 < \pi$, $\tau_1 \neq \tau_0$, which more clearly exhibits the main ideas of exploiting suitable asymptotic properties of solutions. Afterwards, we demonstrate the modifications required in order to cover the case $\omega = 0$. 

\subsection*{The case $\omega \in (0, \pi)$ and $\tau_1 \in (0, \pi) \setminus \{ \tau_0 \}$.}

\step{1}{The Setting.}

Let $\omega \in (0, \pi)$. We define the map
\begin{align*}
	F: \quad & X_1 \times X_1 \times \R \to X_1 \times X_1,
	\\
	F(w, v, b) := & \svec{
	w - \mathcal{R}_\mu^{\tau_1} (w^3 + 3 u_0 w^2 + 3 u_0^2 w + b \: (u_0 + w) v^2)
	}{
	v - \mathcal{R}_\nu^\omega ( v^3 + b v (u_0+ w)^2 )
	}
\end{align*}
with the convolution operators $\mathcal{R}_\mu^{\tau_1}, \mathcal{R}_\nu^\omega: X_3 \to X_1$ from Definition~\eqref{eq_R-lambda-omega}. Observe that $F$ is well-defined since $u, v, w \in X_1$ implies $uvw \in X_3$. 
Recalling Corollary~\ref{cor_asymptotic} and~\eqref{eq_singleHH}, we have
\begin{align*}
	F(w, v, b) = 0 
	\quad \Leftrightarrow \quad
	(u, v, b) := (u_0 + w, v, b) \text{ satisfies~\eqref{eq_system} with asymptotics~\eqref{eq_asymptotic}}. 
\end{align*}
So we aim to find nontrivial zeros of $F$.
Second, we observe that $F$ has a trivial solution family, that is $F(0, 0, b) = 0$ holds for every $b \in \R$.
Third, $F(\,\cdot\, , b)$ is a compact perturbation of the identity on $X_1 \times X_1$ since the operators $\mathcal{R}_\mu^{\tau_1}, \mathcal{R}_\nu^\omega: X_3 \to X_1$ are compact thanks to Proposition~\ref{prop_cont-comp}~(b). 
Moreover, $F$ is twice continuously Fr\'{e}chet differentiable; we have for $\varphi, \psi \in X_1$ and $b \in \R$, denoting by $D$ the Fr\'{e}chet derivative w.r.t. the $w$ and $v$ components, 
\begin{align}\label{eq_diff1}
	D F(0, 0, b)[(\varphi, \psi)] 
	&= \svec{\varphi}{\psi} - \svec{
	3 \: \mathcal{R}_\mu^{\tau_1} (u_0^2 \: \varphi) 
	}{
	b \: \mathcal{R}_\nu^\omega ( u_0^2 \: \psi )
	} =
	\svec{
	\varphi - 3 \: \mathbf{R}_\mu^{\tau_1} \varphi 
	}{
	\psi - b \: \mathbf{R}_\nu^\omega \psi
	}
\end{align}
with compact linear operators $\mathbf{R}_\mu^{\tau_1}, \mathbf{R}_\nu^\omega: X_1 \to X_1$ as in equation~\eqref{eq_R-lin}. We deduce that, due to~\eqref{eq_nondegen-tau0} and $\tau_1 \neq \tau_0$, 
$D F(0, 0, b)[\varphi, \psi]  = 0$ implies $\varphi = 0$. 
So nontrivial elements of $\ker D F(0, 0, b)$ are of the form $(0, \psi)$ where $\psi$ satisfies $\psi = b \: \mathbf{R}_\nu^\omega \psi$.
Proposition~\ref{prop_spectrum} reveals that such nontrivial $\psi$ exists if and only if $b = b_k(\omega, \nu, u_0^2)$, i.e. $\omega_\nu(b \, u_0^2) = k\pi + \omega$ for some $k \in \Z$, and that the associated eigenspaces are one-dimensional. We denote $b_k(\omega)$ instead of $b_k(\omega, \nu, u_0^2)$.
Thus $b \in \{ b_k(\omega) \: | \: k \in \Z \}$ is a necessary condition for bifurcation of solutions of $F(w, v, b) = 0$ from $(0, 0, b)$.
We show in the following that it is also sufficient.

\step{2}{Local Bifurcation.}

We apply the Crandall-Rabinowitz Bifurcation Theorem and, to this end, verify its simplicity and transversality assumptions at the point $(0, 0, b_k(\omega))$. As $F(\,\cdot\, ,b)$ is a compact perturbation of the identity on $X_1 \times X_1$, the Riesz-Schauder Theorem implies that $D F(0, 0, b_k(\omega))$ is a Fredholm operator of index zero. By the previous step, 
\begin{align*}
	\ker D F(0, 0, b_k(\omega)) = \text{span } \left\{ \svec{0}{\psi_k} \right\}
\end{align*}
for some $\psi_k \in X_1 \setminus \{ 0 \}$. To see that the transversality condition holds, we first compute
\begin{align*}
	\partial_b D F(0, 0, b_k(\omega)) [(0, \psi_k)] 
	\overset{\eqref{eq_diff1}}{=} - \svec{0}{\mathbf{R}_\nu^\omega \psi_k} 
	= -\frac{1}{b_k(\omega)} \svec{0}{\psi_k}.
\end{align*}
Then, assuming there was $v \in X_1$ with $v - b_k(\omega) \: \mathbf{R}_\nu^\omega v = \psi_k$, we conclude
\begin{align*}
	v \in \ker (I - b_k(\omega) \, \mathbf{R}_\nu^\omega)^2 \setminus \ker (I - b_k(\omega) \, \mathbf{R}_\nu^\omega),
\end{align*}
which contradicts the algebraic simplicity of the eigenvalue $b_k(\omega)^{-1}$ of $\mathbf{R}_\nu^\omega$ proved in Proposition~\ref{prop_spectrum}. Thus $\partial_b D F(0, 0, b_k(\omega)) [(0, \psi_k)]  \not\in \text{ran } D F(0, 0, b_k(\omega))$,
and the Crandall-Rabinowitz Bifurcation Theorem provides the curve of solutions of $F(w, v, b) = 0$ as described in (ii). We remark that it is smooth since $F$ is of class $C^\infty$. Further, possibly shrinking the neighborhood where the local result holds, we may w.l.o.g. assume fully nontrivial solutions $(u_0 + w, v)$ of~\eqref{eq_system} since the direction of bifurcation with respect to $X_1 \times X_1$ is given by $(0, \psi_k)$.

\step{3}{Global Bifurcation.}

We have already seen that $F(\,\cdot\, , b)$, $b \in \R$, is a compact perturbation of the identity on $X_1 \times X_1$. 
Thus the application of Rabinowitz' Global Bifurcation Theorem only requires to verify that the index of $F(\,\cdot\, , b)$ in $(0, 0)$ changes sign at each value $b = b_k(\omega)$, $k \in \Z$. By the identity~\eqref{eq_diff1}, for $b \not\in \{ b_k(\omega) \: | \: k \in \Z \}$,
\begin{align*}
	\text{ind}_{X_1 \times X_1} \big( F(\,\cdot\, , b), (0, 0) \big)
	&= 
	\text{ind}_{X_1 \times X_1} \big( DF(0, 0, b), (0, 0) \big) 
	\\
	&\overset{\eqref{eq_diff1}}{=} 
	\text{ind}_{X_1}  \big( I - 3 \:  \mathbf{R}_\mu^{\tau_1}, 0 \big) \cdot 
	\text{ind}_{X_1} \big(I - b \: \mathbf{R}_\nu^\omega, 0 \big),
\end{align*}
and hence $\text{ind}_{X_1 \times X_1} \big( F(\,\cdot\, , b), (0, 0) \big)$ changes sign at $b = b_k(\omega)$ if and only if so does 
$\text{ind}_{X_1} \big(I - b \: \mathbf{R}_\nu^\omega, 0 \big)$. The latter change of index occurs since $b_k(\omega)$ is an isolated eigenvalue of algebraic multiplicity 1 of $\mathbf{R}_\nu^\omega$, see Proposition~\ref{prop_spectrum}.

Notice that, by Step 2, $(u_0, 0, b_k(\omega)) \in \overline{\mathcal{S}(\omega)}$, 
and  the Global Bifurcation Theorem by Rabinowitz asserts that the associated connected component $\mathcal{C}_k(\omega)$ of $\overline{\mathcal{S}(\omega)}$  is unbounded or returns to the trivial branch at some point $(u_0, 0, b^\ast) \in \mathcal{T}_{u_0}$. We prove that, in any case, the component is unbounded. 

To see this, we recall the asymptotic phase $\omega_\nu$ as introduced in Proposition~\ref{prop_pruefer}. It satisfies $\omega_\nu(b_k(\omega) u_0^2) = \omega + k \pi$ by definition of $b_k(\omega)$, see Step 1, as well as $\omega_\nu(v^2 + b u^2) \in \omega + \pi \Z$ for all $(u, v, b) \in \mathcal{C}_k(\omega)$ with $v \neq 0$ due to~\eqref{eq_asymptotic} - indeed, we recall here that Proposition~\ref{prop_pruefer} rules out the case that $v \neq 0$ with $v(x) = O\left(\frac{1}{|x|^2}\right)$ as $|x| \to \infty$. 

So if all elements $(u, v, b) \in \mathcal{C}_k(\omega) \setminus \mathcal{T}_{u_0}$ satisfy $v \neq 0$, then as a consequence of the continuity of $\omega_\nu$ as stated in Proposition~\ref{prop_aspt-phase-continuous} and of the fact that $\mathcal{C}_k(\omega)$ is connected by definition, we infer that $\omega_\nu(v^2 + b u^2) = \omega + k \pi$ for all $(u, v, b) \in \mathcal{C}_k(\omega)$.
Let us now assume that $\mathcal{C}_k(\omega)$ returns to the trivial family in some point $(u_0, 0, b^\ast) \in \mathcal{T}_{u_0}$, $b^\ast \neq b_k(\omega)$. Then $\omega_\nu(b^\ast u_0^2) \neq \omega + k \pi$, hence $(u, v, b) \mapsto \omega_\nu(v^2 + b u^2)$ is not constant on $\mathcal{C}_k(\omega)$. Thus, there exists a semitrivial element $(u_1, 0, b_1) \in \mathcal{C}_k(\omega) \setminus \mathcal{T}_{u_0}$, $u_1 \neq u_0$. Since $\mathcal{C}_k(\omega)$ is maximal connected, it contains the unbounded semitrivial family $\mathcal{T}_{u_1} = \{ (u_1, 0, b) \: | \: b \in \R \}$.

\subsection*{The case $\omega = 0$ and $\tau_1 \in (0, \pi) \setminus \{ \tau_0 \}$.}

\step{1}{The Setting.}

We recall that, in case $\omega = 0$, the map $F$ resp. $\mathcal{R}_\nu^\omega$ is not well-defined due to the pole of the cotangent.
We replace it by
\begin{align*}
	G_\sigma: \quad & X_1 \times X_1 \times \R \to X_1 \times X_1,
	\\
	G_\sigma(w, v, b) := & \svec{
	w - \mathcal{R}_\mu^{\tau_1} (w^3 + 3 u_0 w^2 + 3 u_0^2 w + b \: (u_0 + w) v^2)
	}{
	v - \Psi_\nu \ast [v \, (v^2 + b \, (w + u_0)^2)] - (\alpha^{(\nu)}(v) + \sigma \beta^{(\nu)}(v)) \cdot \tilde{\Psi}_\nu 
	}
\end{align*}
with the functionals $\alpha^{(\nu)}, \beta^{(\nu)}$ as in Corollary~\ref{cor_alpha-beta} and for $\sigma = \pm 1$. We will prove the local bifurcation result for each map $G_\sigma$ but, in order to find global bifurcation, we require $G_+$ resp. $G_-$ in order to verify the change of the index at $(0, 0, b)$ with $b \geq 0$ resp. $b \leq 0$. 
Part (b) of that Corollary states that $G_\sigma(w, v, b) = 0$ if and only if the point $(u_0 + w, v, b)$ solves the nonlinear Helmholtz system~\eqref{eq_system} with asymptotics~\eqref{eq_asymptotic}, $\omega = 0$. In particular, $G_+(w, v, b) = 0$ if and only if $G_-(w, v, b) = 0$. Due to~\eqref{eq_nondegen-tau0}, $(\varphi, \psi) \in \ker DG(0, 0, b)$ if and only if 
\begin{align*}
	\varphi \equiv 0, 
	\quad
	- \Delta \psi - \nu \psi = b \: u_0^2 \, \psi, 
	\: \: \psi(x) = c_\psi \frac{\sin(|x|\sqrt{\nu})}{|x|} + O\left(\frac{1}{|x|^2}\right)
\end{align*}
for some $c_\psi \in \R \setminus \{ 0 \}$.
Propositions~\ref{prop_pruefer}~and~\ref{prop_asymptoticphase} thus tell us that a nontrivial solution $\psi = \psi_k \in X_1$ exists if and only if the asymptotic phase satisfies $\omega_\nu(b \, u_0^2) \in \pi \Z$, equivalently $b = b_k(0, \nu, u_0^2) =: b_k(0)$ for some $k \in \Z$, and that the eigenspace is one-dimensional. 
Thus solutions of~\eqref{eq_system},~\eqref{eq_asymptotic} for $\omega = 0$, bifurcate from a point $(u_0, 0, b) \in \mathcal{T}_{u_0}$ only if $b = b_k(0)$ for some $k \in \Z$. 
We show that it happens indeed by checking the assumptions of the Crandall-Rabinowitz Theorem.

\step{2}{Local Bifurcation.}

First, we infer that, since $G_\sigma(\,\cdot\, , b)$ is a compact perturbation of the identity, $DG_\sigma(0, 0, b_k(0))$ is a $1-1-$Fredholm operator. It remains to check transversality. We compute
\begin{align*}
	\partial_b DG_\sigma(0, 0, b_k(0))[0, \psi_k] = - \svec{0}{\Psi_\nu \ast[u_0^2 \, \psi_k]}
\end{align*}
and assume by contradiction that there exist $\varphi, \psi \in X_1$ with $DG_\sigma(0, 0, b_k(0))[(\varphi, \psi)] = \partial_b DG_\sigma(0, 0, b_k(0))[(0, \psi_k)]$. Then $\varphi = 0$ due to~\eqref{eq_nondegen-tau0}, and 
\begin{align}\label{eq_proof-psi}
	\psi = b_k(0) \: \Psi_\nu \ast [u_0^2 \, \psi] + (\alpha^{(\nu)}(\psi) 
	+ \sigma \beta^{(\nu)}(\psi)) \cdot \tilde{\Psi}_\nu - \Psi_\nu \ast[u_0^2 \, \psi_k].
\end{align}
Thus, applying the functional $\alpha^{(\nu)}$ to~\eqref{eq_proof-psi}, we find
\begin{align*}
	\alpha^{(\nu)}(\psi) 
	&= b_k(0) \alpha^{(\nu)}(\Psi_\nu \ast [u_0^2 \, \psi]) +  (\alpha^{(\nu)}(\psi) 
	+ \sigma \beta^{(\nu)}(\psi)) \cdot \alpha^{(\nu)}(\tilde{\Psi}_\nu) - \alpha^{(\nu)}(\Psi_\nu \ast[u_0^2 \, \psi_k])
	\\
	&\overset{\eqref{eq_alpha-beta-coeff}}{=} \alpha^{(\nu)}(\psi) + \sigma \beta^{(\nu)}(\psi)
\end{align*}
and thus, since $\sigma \neq 0$, we conclude $\beta^{(\nu)}(\psi) = 0$. Equation~\eqref{eq_proof-psi} and $DG_\sigma(0, 0, b_k(0))[(0, \psi_k)] = (0, 0)$ further provide, due to Proposition~\ref{prop_cont-comp}~(c), the differential equations
\begin{align}\label{eq_proof-psikpsi}
	- \psi_k'' - \frac{2}{r} \psi_k' - \nu \psi_k = b_k(0) \: u_0^2(r) \, \psi_k,
	\quad
	- \psi'' - \frac{2}{r} \psi' - \nu \psi = b_k(0) \: u_0^2(r) \, \psi - u_0^2 \, \psi_k
\end{align}
for $r > 0$. Moreover, due to $\beta^{(\nu)}(\psi) = 0$ as shown above and $\beta^{(\nu)}(\psi_k) = 0$ which holds since $DG_\sigma(0, 0, b_k(0))[(0, \psi_k)] = (0, 0)$ as explained in Step 1, Proposition~\ref{prop_cont-comp}~(d) yields for the profiles
\begin{equation}\label{eq_proof-psikpsi-asymptotic}
\begin{split}
	&\psi_k(r) = c_k \cdot \frac{\sin(r\sqrt{\nu})}{r} + O\left(\frac{1}{r^2}\right),
	\quad
	\psi_k'(r) = c_k \sqrt{\nu} \cdot \frac{\cos(r\sqrt{\nu})}{r} + O\left(\frac{1}{r^2}\right),
	\\
	&\psi(r) = c \cdot \frac{\sin(r\sqrt{\nu})}{r} + O\left(\frac{1}{r^2}\right),
	\quad
	\psi'(r) = c \sqrt{\nu} \cdot \frac{\cos(r\sqrt{\nu})}{r} + O\left(\frac{1}{r^2}\right)
\end{split}
\end{equation}
as $r \to \infty$ for some $c \in \R$, $c_k \in \R \setminus \{ 0 \}$.
Multiplying the differential equations~\eqref{eq_proof-psikpsi} by $\psi$ resp. $\psi_k$ and taking the difference yields
$(r^2 (\psi_k \psi' - \psi \psi_k'))' = r^2 u_0^2(r) \psi_k^2$, hence for $R > 0$
\begin{align*}
	\int_0^R r^2 u_0^2(r) \psi_k^2(r) \: \mathrm{d}r
	=
	R^2 \left( \psi_k(R) \psi'(R) - \psi(R) \psi_k'(R) \right) 
	\overset{\eqref{eq_proof-psikpsi-asymptotic}}{=} 
	O\left(\frac{1}{R}\right).
\end{align*}
Thus letting $R \nearrow \infty$, we infer $u_0 \psi_k \equiv 0$, a contradiction. Hence 
\begin{align*}
	\partial_b DG_\sigma(0, 0, b_k(0))[(0, \psi_k)] \not \in \mathrm{ran} \, DG_\sigma(0, 0, b_k(0)),
\end{align*}
as asserted, proving transversality and thus bifurcation from a simple eigenvalue.

\step{3}{Global Bifurcation.}

Having already mentioned that $G_\sigma(\,\cdot\, , b)$ is a compact perturbation of the identity on $X_1 \times X_1$, Rabinowitz' Global Bifurcation Theorem applies and yields unbounded connected components $\mathcal{C}_k(0) \subseteq \overline{\mathcal{S}(0)}$ once we show that the index 
\begin{align*}
	\text{ind}_{X_1 \times X_1} \big( G_\sigma(\,\cdot\, , b), (0, 0) \big)
	&= 
	\text{ind}_{X_1 \times X_1} \big( DG_\sigma(0, 0, b), (0, 0) \big) 
	\\	
	&= 
	\text{ind}_{X_1}  \big( I - 3 \:  \mathbf{R}_\mu^{\tau_1}, 0 \big) \cdot 
	\text{ind}_{X_1} \big(I - K_b, 0 \big)
	\\
	\text{where we let } K_b &:= b \: \Psi_\nu \ast [u_0^2 \, \cdot \,] + (\alpha^{(\nu)} + \sigma \beta^{(\nu)}) \cdot \tilde{\Psi}_\nu
\end{align*}
changes sign at $b = b_k(0)$, $k \in \Z$ for a suitable choice of $\sigma \in \{ -1, +1 \}$. More precisely, we analyze bifurcation at $b_k(0) \geq 0$ using the map $G_+$ and at $b_k(0) < 0$ using $G_-$. 

In the following, we verify that $1$ is an algebraically simple eigenvalue of $K_{b_k(0)}$ and, moreover, the corresponding perturbed eigenvalue $\lambda_b \approx 1$ of $K_b$ for $b \approx b_k(0)$ has the property that $\lambda_b - 1$ changes sign as $b$ crosses $b_k(0)$. For the existence, algebraic simplicity and continuous dependence of the perturbed eigenvalue $\lambda_b$ on $b$ we refer to Kielhöfer's book~\cite{kielhoefer}, p.~203. Rabinowitz' Global Bifurcation Theorem in the version of~\cite{kielhoefer},~Theorem~II.3.3 then applies, and unboundedness of the component can then be proved as in Step 3 above.

\subsubsection*{Algebraic Simplicity.}
Here we adapt the proof of algebraic simplicity in Proposition~\ref{prop_spectrum} to the case $\omega = 0$ resp. to the map $G_\sigma$.
Let us assume that $\ker (I - K_{b_k(0)}) = \text{span} \{ w \}$ and $v \in \ker (I - K_{b_k(0)})^2 \setminus \ker (I - K_{b_k(0)})$. Then $v - K_{b_k(0)}v \in \ker (I - K_{b_k(0)})$, and without loss of generality, we have $v - K_{b_k(0)} v = w = K_{b_k(0)}w$, hence
\begin{equation}\label{eq_proof-wvConvolution}
\begin{split}
	w &= b_k(0) \: \Psi_\nu \ast [u_0^2 \, w] + (\alpha^{(\nu)}(w) + \sigma \beta^{(\nu)}(w)) \cdot \tilde{\Psi}_\nu,
	\\
	v &= b_k(0) \: \Psi_\nu \ast [u_0^2 \, (v + w)] + (\alpha^{(\nu)}(v+w) + \sigma \beta^{(\nu)}(v+w)) \cdot \tilde{\Psi}_\nu.
\end{split}
\end{equation}
Corollary~\ref{cor_alpha-beta} implies that the profiles satisfy
\begin{align}\label{eq_proof-wvODE}
	- w'' - \frac{2}{r} w' - \nu w = b_k(0) u_0^2 \: w,
	\quad
	- v'' - \frac{2}{r} v' - \nu v = b_k(0) u_0^2 \: (v + w)
	\qquad 
	\text{on } \R^3
\end{align}
as well as $\beta^{(\nu)}(w) = 0$. Applying $\alpha^{(\nu)}$ to the second identity in~\eqref{eq_proof-wvConvolution} and recalling the identities~\eqref{eq_alpha-beta-coeff}, we further have $\beta^{(\nu)}(v) = - \sigma \alpha^{(\nu)}(w)$. The asymptotic expansions in Proposition~\ref{prop_cont-comp}~(d) imply
\begin{equation}\label{eq_proof-wvAsymptotic}
\begin{split}
	w(r) &= \alpha^{(\nu)}(w) \: \frac{\sin(r \sqrt{\nu})}{r} + O\left(\frac{1}{r^2}\right),
	\\
	w'(r) &= \alpha^{(\nu)}(w) \sqrt{\nu} \: \frac{\cos(r \sqrt{\nu})}{r} + O\left(\frac{1}{r^2}\right),
	\\
	v(r) &= \alpha^{(\nu)}(v) \: \frac{\sin(r \sqrt{\nu})}{r} - \sigma \alpha^{(\nu)}(w) \: \frac{\cos(r \sqrt{\nu})}{r} 
	+ O\left(\frac{1}{r^2}\right),
	\\
	v'(r) &= \alpha^{(\nu)}(v) \sqrt{\nu} \: \frac{\cos(r \sqrt{\nu})}{r} + \sigma \alpha^{(\nu)}(w) \sqrt{\nu} \: \frac{\sin(r \sqrt{\nu})}{r} 
	+ O\left(\frac{1}{r^2}\right).
\end{split}
\end{equation}
For $r \geq 0$, we introduce $q(r) := r^2 \left(w(r) v'(r) - v(r) w'(r) \right)$. Using the differential equations in~\eqref{eq_proof-wvODE}, we find after a short calculation
\begin{align*}
	q'(r) = - r^2 b_k(0) u_0^2(r) w^2(r)
	\qquad (r > 0),
\end{align*}
hence $q$ is nondecreasing if $b_k(0)\leq 0$ and nonincreasing if $b_k(0) \geq 0$. On the other hand, $q(0) = 0$, and the asymptotic expansions~\eqref{eq_proof-wvAsymptotic} imply as $r \to \infty$
\begin{align*}
	q(r) = \sigma \cdot \alpha^{(\nu)}(w)^2 \sqrt{\nu} + O\left(\frac{1}{r}\right).
\end{align*}
Since $\alpha^{(\nu)}(w) \neq 0$ according to Proposition~\ref{prop_pruefer}, and since we choose $\sigma = +1$ to discuss $b_k(0) \geq 0$ and $\sigma = -1$ for $b_k(0) < 0$, this contradicts the monotonicity derived before. Hence $\ker (I - K_{b_k(0)}) = \ker (I - K_{b_k(0)})^2$, as claimed.

\subsubsection*{Perturbation of the eigenvalue.}
We now discuss the perturbation of the simple eigenvalue $\lambda_{b_k(0)} = 1$ of $K_{b_k(0)}$. 
Throughout the following lines, we consider a perturbed value $b \approx b_k(0), b \neq b_k(0)$ and the corresponding eigenpair $\lambda_b \approx 1$ and $w_b \in X_1$ with $K_b w_b = \lambda_b w_b$. It satisfies
\begin{align}\label{eq_proof-crossing}
	- \Delta w_b - \nu w_b = \frac{b}{\lambda_b} u_0^2(x) \: w_b \quad \text{on } \R^3,
	\qquad
	(\lambda_b - 1) \alpha^{(\nu)}(w_b) = \sigma \beta^{(\nu)}(w_b).
\end{align}
This immediately implies that $\lambda_b \neq 1$ since  $b \approx b_k(0), b \neq b_k(0)$ and hence $\beta^{(\nu)}(w_b) \neq 0$ due to the strict monotonicity of the asymptotic phase, see Proposition~\ref{prop_asymptoticphase}. 
We recall that the asymptotic phase of $w_b$ satisfies 
\begin{align}\label{eq_proof-crossingAsympPh}
	\omega_\nu(b_k(0) u_0^2) \in \pi \Z
	\qquad\text{and}\qquad
	\frac{\alpha^{(\nu)}(w_b)}{\beta^{(\nu)}(w_b)} = \cot(\omega_\nu(b \lambda_b^{-1} u_0^2))
	\quad (b \neq b_k(0), b \approx b_k(0)).
\end{align}
The latter is a consequence of various representation results applied to~\eqref{eq_proof-crossing},
\begin{align*}
	w_b(x) &\overset{\text{Cor. }\ref{cor_alpha-beta}, \eqref{eq_alpha-beta-coeff}}{=} 
	\beta^{(\nu)}(w_b) \cdot \Psi_\nu(x) + \alpha^{(\nu)}(w_b) \cdot \tilde{\Psi}_\nu (x) + O\left(\frac{1}{|x|^2}\right),
	\\
	w_b(x) &\overset{\text{Cor. }\ref{cor_asymptotic}}{=} 
	\left( \Psi_\nu + \cot(\omega_\nu(b \lambda_b^{-1} u_0^2)) \cdot \tilde{\Psi}_\nu \right)
	\ast \left[ \frac{b}{\lambda_b} u_0^2(x) \: w_b \right]
	\\
	&\overset{\text{Prop. }\ref{prop_cont-comp}}{=}
	4 \pi \sqrt{\frac{\pi}{2}} \cdot \frac{b}{\lambda_b} \widehat{u_0^2 w_b}(\sqrt{\nu}) \cdot
	\left( \Psi_\nu(x) + \cot(\omega_\nu(b \lambda_b^{-1} u_0^2)) \cdot \tilde{\Psi}_\nu(x) \right) + O\left(\frac{1}{|x|^2}\right).
\end{align*}
We now discuss the values $b_k(0) \geq 0$, i.e. $\sigma = +1$.
In case $b > b_k(0)$ we show that $\lambda_b > 1$. Assuming $\lambda_b < 1$, we infer from the second identity in~\eqref{eq_proof-crossing} that $\text{sgn }\alpha^{(\nu)}(w_b) \neq \text{sgn } \beta^{(\nu)}(w_b)$ and thus $\omega_\nu(b \lambda_b^{-1} u_0^2) \in \left( -\frac{\pi}{2}, 0 \right) + \pi \Z$ due to~\eqref{eq_proof-crossingAsympPh}. But since $b \lambda_b^{-1} > b_k(0)$, the monotonicity stated in Proposition~\ref{prop_asymptoticphase} implies $\omega_\nu(b \lambda_b^{-1} u_0^2) \in \omega_\nu(b_k(0) u_0^2) + \left(0, \frac{\pi}{2} \right) \subseteq \left(0, \frac{\pi}{2} \right) + \pi\Z $, a contradiction. Hence, as claimed, $\lambda_b > 1$.
In the same way, for $b < b_k(0)$, we can show that $\lambda_b < 1$. 
Following the same strategy, we see for $b_k(0) < 0$, i.e. $\sigma = -1$, that $b > b_k(0)$ implies $\lambda_b < 1$ and $b < b_k(0)$ implies $\lambda_b > 1$.

We have thus proved that, as $b$ crosses $b_k(0)$, the perturbed eigenvalue $\lambda_b$ crosses $\lambda_{b_k(0)} = 1$ and hence the sign of the Leray-Schauder index 
$\text{ind}_{X_1 \times X_1} \big( G_{\sigma}(\,\cdot\, , b), (0, 0) \big)$ changes at $b = b_k(0)$ for all $k \in \Z$ and for $\sigma \in \{ \pm 1 \}$ chosen as above.

\subsection*{The case $\tau_1 = 0$.}

This is covered by redefining the first components of $F$ resp. $G_\sigma$, 
\begin{align*}
	&(w, v, b) \mapsto &&w - \Psi_\mu \ast [w^3 + 3 u_0 w^2 + 3 u_0^2 w + b \: (u_0 + w) v^2] 
	- \left [ \alpha^{(\mu)} (w) + \beta^{(\mu)}(w) \right] \cdot \tilde{\Psi}_\mu 
	\\ & &&=: h(w, v, b)
	\\
	&\text{instead of }
	\\
	&(w, v, b) \mapsto &&w - \mathcal{R}_\mu^{\tau_1} (w^3 + 3 u_0 w^2 + 3 u_0^2 w + b \: (u_0 + w) v^2).
\end{align*}
This redefinition is similar to the changes in the second component when passing from $F$ resp. parameters $\omega \in (0, \pi)$ to $G_\sigma$ suitable for $\omega = 0$. Then still, $F$ resp. $G_\sigma$ is a compact perturbation of the identity. The redefinition ensures that, due to Part~(b) of Corollary~\ref{cor_alpha-beta}, $h(w,v,b) = 0$ implies that 
\begin{align*}
	&- \Delta w - \mu w = (u_0 + w)^3 - u_0^3 + b \: (u_0 + w) v^2 \quad \text{on } \R^3,
	\\
	&w(x) = c_w \frac{\sin(|x|\sqrt{\mu})}{|x|} + O\left(\frac{1}{|x|^2}\right)
	\quad \text{as } |x| \to \infty
\end{align*}
for some $c_w \in \R$, i.e. that the $w$ component of zeros of $F$ resp. $G_\sigma$ satisfies~\eqref{eq_system},~\eqref{eq_asymptotic}. Similarly, for $\varphi, \psi \in X_1$ with $Dh(0, 0, b)[(\varphi, \psi)] = (0, 0)$, we obtain
\begin{align*}
	&- \Delta \varphi - \mu \varphi = 3 u_0^2(x) \varphi  \quad \text{on } \R^3,
	\\
	&\varphi(x) = c_\varphi \frac{\sin(|x|\sqrt{\mu})}{|x|} + O\left(\frac{1}{|x|^2}\right)
	\quad \text{as } |x| \to \infty
\end{align*}
for some $c_\varphi \in \R \setminus \{ 0 \}$, which implies $\varphi = 0$ thanks to the nondegeneracy condition~\eqref{eq_nondegen-tau0}. These are the only properties of the first component of $F$ resp. $G_\sigma$ required in the proof for $\tau_1 \neq 0$, which we can now again follow line by line, closing the proof of Theorem~\ref{THEtheorem}. \hfill $\square$

\subsection*{Proof of Remark~\ref{THEremark}}

\begin{itemize}
\item[(a)]
The Steps~1 of the proof above in fact show that solutions of~\eqref{eq_system},~\eqref{eq_asymptotic} bifurcate from $(u_0, 0, b) \in \mathcal{T}_{u_0}$ only if $b = b_k(\omega)$ for $k \in \Z$; Steps~2 show that this condition is also sufficient.
\item[(b)]
By Proposition~\ref{prop_asymptoticphase}, the map $q: \R \to \R, \: q(b) := \omega_\nu (b \, u_0^2)$ is strictly increasing and onto. Having chosen $b_k(\omega) = q^{-1}( \omega + k \pi)$ for $\omega \in [0, \pi), k \in \Z$, we infer strict monotonicity and surjectivity of the map $\R \to \R, \:	\omega + k \pi \mapsto b_k(\omega)$.
\item[(c)]
In  Steps~2 we have seen that in a neighborhood of the bifurcation point $(u_0, 0, b_k(\omega))$, the continuum  $\mathcal{C}_k(\omega)$ contains only fully nontrivial solutions apart from $(u_0, 0, b_k(\omega))$ itself. 
Following the argumentation which was given in detail for the case $\omega \in (0, \pi)$ at the end of Step~3 (and also holds for $\omega = 0$), we infer for all $(u, v, b) \in \mathcal{C}_k(\omega)$ from this neighborhood that the asymptotic phase of $v$ satisfies $\omega_\nu(v^2 + b u^2) = \omega + k \pi$. More generally, $\omega_\nu(v^2 + b u^2) = \omega + k \pi$ holds on every connected subset of $\mathcal{C}_k(\omega)$ containing $(u_0, 0, b_k(\omega))$ but no other semitrivial solution with $v = 0$.
\item[(d)]
Assuming $\tau_1 \neq \sigma_0$, any solution $(u, v, b)$ of~\eqref{eq_system},~\eqref{eq_asymptotic} satisfies 
\begin{align*}
	u(x) = u_0(x) + w(x)
	= c_0 \: \frac{\sin(|x| \sqrt{\mu} + \sigma_0)}{|x|} + c_w \: \frac{\sin(|x| \sqrt{\mu} + \tau_1)}{|x|} + O\left(\frac{1}{|x|^2}\right) 
\end{align*} 
as $|x| \to \infty$ for some $c_w \in \R, c_0 \in \R \setminus \{ 0 \}$ by Proposition~\ref{prop_u0-sigma0-tau0} and by the asymptotic condition~\eqref{eq_asymptotic}. Hence, comparing the leading-order terms, we see that $u \neq 0$.
Moreover, as recalled in~(b), the values $b_k(\omega) = q^{-1}( \omega + k \pi)$ do not change when choosing another asymptotic parameter $\tau_1$ in~\eqref{eq_asymptotic}.
\hfill $\square$
\end{itemize}

\section{Proof of Theorem~\ref{THEtheorem_diagonal}}\label{sect_prf2}

We now prove the occurence of bifurcations from the diagonal solution family
\begin{align*}
	\mathfrak{T}_{u_0} := \left\{
		(u_b, u_b, b) \: \bigg| \: u_b = \frac{1}{\sqrt{1+b}} \: u_0, \: b > -1
	\right\}
\end{align*}
as stated in Theorem~\ref{THEtheorem_diagonal}. To this end we first rewrite
the system~\eqref{eq_system} in an equivalent but more convenient way. Looking for solutions $(u,v,b) \in X_1\times X_1\times\R \setminus \mathfrak T$, we introduce the functions $w_1, w_2 \in X_1$ via
\begin{align*}
	u =: u_b+w_1-w_2, \qquad v =:u_b +w_1+w_2.
\end{align*} 
A few computations then yield that bifurcation at the point $(u_b,u_b,b)$ occurs if and only if we have bifurcation from the trivial solution of the nonlinear Helmholtz system
\begin{align}\label{eq_system-w12}
	\begin{cases}
  -\Delta w_1 - \mu w_1 = (1+b)\big( (w_1+u_b)^3-u_b^3 \big) + (3-b)(w_1+u_b)w_2^2
  & \text{on } \R^3, \\
  -\Delta w_2 - \mu w_2 = (1+b) w_2^3 + (3-b)(w_1+u_b)^2w_2
  & \text{on } \R^3,
  \end{cases}
\end{align} 
and the asymptotic conditions~\eqref{eq_asymptotic_diagonal} are equivalent to
\begin{align}\label{eq_asymptotic-w12}
	w_1(x) = c_1 \: \frac{\sin(|x| \sqrt{\mu} + \tau_1)}{|x|} + O\left(\frac{1}{|x|^2}\right),
	\quad
	w_2(x) = c_2 \: \frac{\sin(|x| \sqrt{\mu} + \omega)}{|x|} + O\left(\frac{1}{|x|^2}\right)
\end{align}
as $|x| \to \infty$ for some $c_1, c_2 \in \R$.
As in the proof of Theorem~\ref{THEtheorem}, the functional analytical setting in the special cases $\omega=0$ or $\tau_1=0$ is different from the general one since a substitute for the operators $\mathcal{R}_\mu^{\tau_1}, \mathcal{R}_\mu^\omega$ has to be found, see the definition of $G_\sigma$ in the proof of Theorem~\ref{THEtheorem}.
In order to keep the presentation short we only discuss the case $\tau_1, \omega \in (0, \pi)$ and refer to the proof of Theorem~\ref{THEtheorem} for the modifications in the remaining cases. So we introduce the map 
$ F: X_1\times X_1\times (-1, \infty) \to X_1\times X_1$ via
\begin{align*}
	F(w_1,w_2,b) := 
   \svec{w_1}{w_2} - 
  	\svec{\mathcal{R}_\mu^{\tau_1} 
  	\Big( (1+b)\big( (w_1+u_b)^3-u_b^3 \big) + (3-b)(w_1+u_b)w_2^2  \Big)}
  	{\mathcal{R}_\mu^{\omega}\Big( (1+b) w_2^3 + (3-b)(w_1+u_b)^2w_2  \Big)}.
\end{align*}
Then $F(0,0,b)=0$ for all $b>-1$, $F(\,\cdot\, , b)$ is a compact perturbation of the identity on $X_1 \times X_1$ and it remains to find bifurcation points for this equation.
First we identify candidates for bifurcation points, so we first compute those $b\in (-1, \infty)$ where $\ker DF(0,0,b)$ is nontrivial. Using 
\begin{align*}
	DF(0,0,b)[(\phi_1,\phi_2)]
  	= \svec{\phi_1}{\phi_2} - \svec{\mathcal{R}_\mu^{\tau_1}\big( 3(1+b)u_b^2\phi_1 \big)}
  	{\mathcal{R}_\mu^{\omega}\big( (3-b)u_b^2\phi_2  \big)} 
   = \svec{\phi_1}{\phi_2} - \svec{3 \mathbf{R}_\mu^{\tau_1} \phi_1}
   {\frac{3-b}{1+b} \cdot \mathbf{R}_\mu^{\omega} \phi_2},
\end{align*}
we get that nontrivial kernels occur exactly if $\frac{3-b}{1+b} = b_k(\omega)$ for some $k\in\Z$, cf. Steps 1 in the previous proof. For the analogous result in the Schrödinger case we refer to Lemma~3.1~\cite{Bartsch2}.  
So we find   
\begin{align*}
  \ker DF(0,0,b) = \text{span }\Big\{\svec{0}{\psi_k} \Big\}
  \quad\text{provided } b= \frac{3-b_k(\omega)}{b_k(\omega)+1} > -1
\end{align*}
for some $\psi_k\in X_1\sm\{0\}$. Notice that the first component of the kernel element is zero by choice of $\tau_1$, see Proposition~\ref{prop_u0-sigma0-tau0}. Using the algebraic simplicity of $\psi_k$ proved in Proposition~\ref{prop_spectrum} we infer exactly as in the proof of Theorem~\ref{THEtheorem} that the transversality condition holds and that the Leray-Schauder index changes at the bifurcation point. So, choosing as $(\tilde{b}_k(\omega))_{k \in \Z}$ the subsequence of $(b_k(\omega))_{k \in \Z}$ with $\frac{3-b_k(\omega)}{b_k(\omega)+1} > -1$, the Crandall-Rabinowitz Bifurcation Theorem and Rabinowitz' Global Bifurcation Theorem yield statements (ii) and (i) of the Theorem, respectively. 

Unboundedness of the components can also be deduced as before. Indeed, assuming that $\mathfrak{C}_k(\omega)$ is unbounded, it returns to $\mathfrak{T}_{u_0}$ at some point $(u_{b^\ast}, u_{b^\ast}, b^\ast) \neq (u_{b_k(\omega)}, u_{b_k(\omega)}, b_k(\omega))$ by Rabinowitz' Theorem. We then infer that the phase $\omega_\nu((1 + b) w_2^2 + (3 - b)(w_1 + u_b)^2)$ cannot be constant along $\mathfrak{C}_k(\omega)$. Due to Proposition~\ref{prop_pruefer} applied to $w_2$ in~\eqref{eq_system-w12}, this requires the existence of some element $(u, v, b) \in \mathfrak{C}_k(\omega)$ with $w_2 = \frac{1}{2} (v - u) = 0$, and hence the associated unbounded diagonal family belongs to $\mathfrak{C}_k(\omega)$.
\hfill $\square$

\section{Proofs of the Results in Section~\ref{sect_scalar}}\label{sect_prfs}

Before proving Proposition~\ref{prop_cont-comp}, we establish two helpful explicit results. 
The first one provides a formula for the Fourier transform of radially symmetric functions, to be found e.g. in~\cite{Stein},~p.~430.
\begin{lem}\label{lem_fourier}
For $f \in X_3$ and $x \in \R^3 \setminus \{ 0 \}$, we have
\begin{align*}
	\hat{f}(x) = \sqrt{\frac{2}{\pi}} \int_0^\infty f(r) \cdot \frac{\sin(|x|r)}{|x|r} \: r^2 \: \mathrm{d}r.
\end{align*}
\end{lem}
We denote by
\begin{align}\label{eq_fundamental-complex}
	\Phi_\lambda(x) := \frac{\mathrm{e}^{\i \sqrt{\lambda} |x|}}{4 \pi |x|}
	\overset{\eqref{eq_fundamental}}{=} \Psi_\lambda(x) + \i \: \tilde{\Psi}_\lambda(x)
	\qquad
	\text{for } \lambda > 0, x \in \R^3 \setminus \{ 0 \}
\end{align} 
a (complex) fundamental solution of the Helmholtz equation $-\Delta \phi - \lambda \phi = 0$ on $\R^3$, and provide the following pointwise formula for convolutions with kernel $\Phi_\lambda$.
\begin{lem}\label{lem_convol}
For $f \in X_3$ and $x \in \R^3 \setminus \{ 0 \}$, we have
\begin{align*}
	(\Phi_\lambda \ast f)(x) 
 	&= \frac{\mathrm{e}^{\i \sqrt{\lambda} |x|}}{ |x| } \cdot  \int_0^{|x|}	\frac{\sin(\sqrt{\lambda} r) }{\sqrt{\lambda} r} 	
 	\cdot f(r) \: r^2 \:  \mathrm{d}r
	+ \frac{\sin(\sqrt{\lambda} |x|)}{|x|} \cdot
	 \int_{|x|}^\infty \frac{\mathrm{e}^{\i \sqrt{\lambda} r}}{\sqrt{\lambda} r} \cdot f(r) \: r^2 	\:  \mathrm{d}r
	\\
	&= \sqrt{\frac{\pi}{2}} \: \hat{f}(\sqrt{\lambda}) 
 	\cdot \frac{\mathrm{e}^{\i \sqrt{\lambda} |x|}}{|x|}
 	+ \int_{|x|}^\infty f(r) \cdot  
 	\frac{\mathrm{e}^{\i\sqrt{\lambda}  r} \sin(\sqrt{\lambda} |x|) - \mathrm{e}^{\i \sqrt{\lambda} |x|} \sin(\sqrt{\lambda} r)}
	{\sqrt{\lambda} r|x|}  \: r^2 \mathrm{d}r.
\end{align*}
\end{lem}

\begin{proof}
Let $f \in X_3$ and $x \neq 0$. Observing that the singular function $\Phi_\lambda$ is locally integrable in $\R^3$, the convolution is well-defined and we compute using spherical coordinates with respect to the $x$ direction
\begin{align*}
	(\Phi_\lambda \ast f)(x) 
	&= \int_{\R^3} \frac{\mathrm{e}^{\i \sqrt{\lambda} |x-y|}}{4 \pi |x-y|} \: f(y) \: \mathrm{d}y
	\\
	&= \int_0^\infty \int_{0}^\pi  \int_0^{2\pi}   
	\frac{\mathrm{e}^{\i \sqrt{\lambda} \sqrt{|x|^2 + r^2 - 2 |x| r \cos(\vartheta)}}}{4 \pi \sqrt{|x|^2 + r^2 - 2 |x| r \cos(\vartheta)}} 
	\: f(r) \: r^2 \sin(\vartheta)
	\: \mathrm{d}\varphi \mathrm{d}\vartheta \mathrm{d}r
	\\
	&= \int_0^\infty \left[
	\frac{\mathrm{e}^{\i \sqrt{\lambda} \sqrt{|x|^2 + r^2 - 2 |x| r \cos(\vartheta)}}}{2\i \sqrt{\lambda}  |x| r}
	\right]_{\vartheta = 0}^{\vartheta = \pi}  \: f(r) \: r^2 
	\:  \mathrm{d}r
	\\
	&= \int_0^\infty 
	\frac{\mathrm{e}^{\i \sqrt{\lambda} ||x|+r|} - \mathrm{e}^{\i \sqrt{\lambda} ||x|- r|}}{2\i \sqrt{\lambda}  |x| r}
	\cdot f(r) \: r^2 
	\:  \mathrm{d}r
	\\
	&= \int_0^{|x|}
	\frac{\mathrm{e}^{\i \sqrt{\lambda} |x|} \cdot \sin(\sqrt{\lambda} r) }{ \sqrt{\lambda} |x| r}
	\cdot f(r) \: r^2 
	\:  \mathrm{d}r
	+ \int_{|x|}^\infty 
	\frac{\mathrm{e}^{\i\sqrt{\lambda} r} \cdot \sin(\sqrt{\lambda} |x|)}{ \sqrt{\lambda} |x| r}
	\cdot f(r) \: r^2 
	\:  \mathrm{d}r.
\end{align*}
When combined with Lemma~\ref{lem_fourier}, this yields the asserted identity. 
\end{proof}

\subsection{Proof of Proposition~\ref{prop_cont-comp}}

We now prove, one by one, the assertions of Proposition~\ref{prop_cont-comp} for convolutions with $\Phi_\lambda$ in place of $\alpha \Psi_\lambda + \tilde{\alpha} \tilde{\Psi}_\lambda$.  The latter (real-valued) case can be deduced from the former using formula~\eqref{eq_fundamental-complex}. Unless stated otherwise, we extend norms defined on spaces of real-valued functions to complex-valued functions $g: \R^3 \to \C$ by considering the respective norm of $|g|: \R^3 \to \R$. 

(a) is a consequence of Theorem~2.1 in~\cite{EvequozWeth}.
The solution properties stated in (c) can be verified by direct computation. We prove continuity and compactness as stated in (b) and the asymptotic properties given in (d).

\step{1}{Proof of (b), first part. Continuity.}

Due to the continuous embedding $X_3 \subseteq \L{\frac{4}{3}}$, see~\eqref{eq_embed}, the convolution is well-defined for $f \in X_3$. Using Young's convolution inequality, we get
\begin{align*}
	|(\Phi_\lambda \ast  f) (x)|	
	&\leq
	| \left( \mathds{1}_{B_1(0)} \Phi_\lambda \right) \ast f (x) | + | \left(  \mathds{1}_{\R^3 \setminus B_1(0)} \Phi_\lambda \right) \ast f (x) |
	\\
	&\leq
	\int_{\R^3} | f(x-y) | \mathds{1}_{B_1(0)}(y) \frac{\mathrm{d}y}{4\pi |y|} 
	+ \norm{\left( \mathds{1}_{\R^3 \setminus B_1(0)} \Phi_\lambda \right) \ast f}_{L^\infty(\R^3)}
	\\
	&\leq
	\norm{f}_{L^\infty(\R^3)} \int_{B_1(0)} \frac{\mathrm{d}y}{4\pi |y|} 
	+ \norm{\mathds{1}_{\R^3 \setminus B_1(0)} \Phi_\lambda}_{L^4(\R^3)} \norm{f}_{L^\frac{4}{3}(\R^3)}
	\\
	&\leq
	\norm{f}_{L^\infty(\R^3)} \int_{B_1(0)} \frac{\mathrm{d}y}{4\pi |y|} 
	+ \norm{f}_{L^\frac{4}{3}(\R^3)}
	\left(\int_{\R^3 \setminus B_1(0)} \frac{\mathrm{d}y}{(4\pi |y|)^4}\right)^\frac{1}{4}
	\\
	&=
	\frac{1}{2} \norm{f}_{L^\infty(\R^3)}
	+ (4\pi)^{-\frac{3}{4}} \norm{f}_{L^\frac{4}{3}(\R^3)}
	\\
	&\leq
	\frac{1}{2} \norm{f}_{X_3}
	+ (4\pi)^{-\frac{3}{4}} \left( \int_{\R^3} \frac{\norm{f}_{X_3}^\frac{4}{3}}{(1 + |y|^2)^2} \: \mathrm{d}y \right)^\frac{3}{4}
	\\
	&\leq
	\left( \frac{1}{2} + \left( \int_0^\infty \frac{1}{1 + r^2} \: \mathrm{d}r \right)^\frac{3}{4} \right) \norm{f}_{X_3}
	\\
	&\leq
	\frac{1 + \pi}{2} \cdot \norm{f}_{X_3}.
\end{align*}
Next, by means of Lemma~\ref{lem_convol}, we estimate for $x \in \R^3 \setminus \{ 0 \}$ in the weighted norm
\begin{align*}
	||x| \cdot (\Phi_\lambda \ast  f) (x)|
	&= \left|
	\mathrm{e}^{\i \sqrt{\lambda} |x|} \cdot  \int_0^{|x|}	\frac{\sin(\sqrt{\lambda} r) }{\sqrt{\lambda} r} 	
 	\cdot f(r) \: r^2 \:  \mathrm{d}r
	+ \sin(\sqrt{\lambda} |x|) \cdot 
	\int_{|x|}^\infty \frac{\mathrm{e}^{\i\sqrt{\lambda} r}}{\sqrt{\lambda} r} \cdot f(r) \: r^2 	\:  \mathrm{d}r
	\right| 
	\\
	&\leq 
	\int_0^{|x|}	\frac{1}{\sqrt{\lambda} r} 	
 	\cdot \frac{\norm{f}_{X_3}}{(1+r^2)^\frac{3}{2}} \: r^2 \:  \mathrm{d}r
	+ 	\int_{|x|}^\infty \frac{1}{\sqrt{\lambda} r} \cdot \frac{\norm{f}_{X_3}}{(1+r^2)^\frac{3}{2}} \: r^2 	\:  \mathrm{d}r
	\\
	&\leq 
	\frac{1}{\sqrt{\lambda}} \cdot \norm{f}_{X_3} \cdot \int_0^\infty \frac{\mathrm{d}r}{1+r^2}
	\\
	&=
	\frac{\pi}{2 \sqrt{\lambda}} \cdot \norm{f}_{X_3}.
\end{align*}
Combining both estimates, we have shown that
\begin{align*}
	\norm{\Phi_\lambda \ast f}_{X_1}
	= \sup_{x \in \R^3} \sqrt{1 + |x|^2} |(\Phi_\lambda \ast f) (x)| 
	\leq \left[  \frac{1 + \pi}{2} + \frac{\pi}{2 \sqrt{\lambda}} \right] \cdot \norm{f}_{X_3}.
\end{align*}

\step{2}{Proof of (d), first part. Asymptotics of $w$.}

For $f \in X_3$ and $r = |x| > 0$, Lemma~\ref{lem_convol} implies
\begin{equation}\label{eq_R(eps)-3}
\begin{split}
	\left| (\Phi_\lambda \ast f)(r) - \sqrt{\frac{\pi}{2}} \: \hat{f}(\sqrt{\lambda}) 
 	\frac{\mathrm{e}^{\i \sqrt{\lambda} r}}{r} \right|
 	& \quad = \left| \int_{r}^\infty f(s)   
 	\frac{\mathrm{e}^{\i\sqrt{\lambda}  s} \sin(\sqrt{\lambda} r) - \mathrm{e}^{\i \sqrt{\lambda} r} \sin(\sqrt{\lambda} s)}
	{\sqrt{\lambda} sr}  \: s^2 \mathrm{d}s \right|
	\\
 	& \quad \leq  \int_{r}^\infty \frac{\norm{f}_{X_3}}{(1+s^2)^\frac{3}{2}} \cdot  
 	\frac{2}	{\sqrt{\lambda} r}  \: s \mathrm{d}s 
 	\\
 	& \quad \leq \norm{f}_{X_3} \cdot \frac{2}{\sqrt{\lambda} r}  \int_{r}^\infty \frac{1}{s^2} \: \mathrm{d}s 
 	\\
 	& \quad = \norm{f}_{X_3} \cdot \frac{2}{\sqrt{\lambda} r^2}.  
\end{split}
\end{equation}
As a consequence, we derive the formula stated for $\tilde{\Psi}_\lambda \ast f$. Due to (c), $\tilde{\Psi}_\lambda \ast f$ is a radial solution of the homogeneous Helmholtz equation $- \Delta w - \lambda w = 0$ on $\R^3$ and hence a scalar multiple of $\tilde{\Psi}_\lambda$ itself. The asymptotics in~\eqref{eq_R(eps)-3} justify the asserted constant.

\step{3}{Proof of (d), second part. Asymptotics of $w'$.}

We let $w := \Phi_\lambda \ast f$, denote by $w'$ the radial derivative, again identifying the radial functions with their profiles. 
We introduce the auxiliary function $z(r) := r \cdot w(r)$ ($r \geq 0$). Then by~(c) $z$ is twice continuously differentiable with
\begin{align*}
	- z'' - \lambda z = r \cdot f(r) \text{  on } (0, \infty),
	\qquad
	z(r) = \sqrt{\frac{\pi}{2}} \hat{f}(\sqrt{\lambda}) \cdot \mathrm{e}^{\i \sqrt{\lambda} r} + O\left( \frac{1}{r} \right) \text{  as } r \to \infty.
\end{align*}
Letting $\delta(r) := z(r) - \sqrt{\frac{\pi}{2}} \hat{f}(\sqrt{\lambda}) \cdot \mathrm{e}^{\i \sqrt{\lambda} r}$, we observe
\begin{align*}
	\delta(r) = O\left(\frac{1}{r}\right), 
	\quad
	\delta''(r) = - \lambda \delta(r) - r \cdot f(r) = O\left(\frac{1}{r}\right)
	\quad
	\text{as } r \to \infty.
\end{align*}
By Taylor's Theorem, for $r > 0$, we find $\tau(r) \in (0, 1)$ with
$\delta (r+1) = \delta (r) + \delta'(r) + \frac{1}{2} \delta''(r + \tau(r))$, whence also $\delta'(r) = O\left(\frac{1}{r}\right)$. In other words,
\begin{align*}
	z'(r) = \i \sqrt{\lambda} \,  \cdot \sqrt{\frac{\pi}{2}} \hat{f}(\sqrt{\lambda}) \cdot \mathrm{e}^{\i \sqrt{\lambda} r} 
	+ O\left(\frac{1}{r}\right) \quad
	\text{as } r \to \infty,
\end{align*}
and this proves the assertion since $w'(r) = \frac{z'(r)}{r} - \frac{z(r)}{r^2} = \frac{z'(r)}{r} + O\left(\frac{1}{r^2}\right)$.

\step{4}{Proof of (b), second part. Compactness.}

We consider a bounded sequence $(f_n)_n$ in the space $X_3$ and aim to prove convergence of a subsequence of $(u_n)_n$ where $u_n := \Phi_\lambda \ast f_n$ in the space $X_1$. 
First, due to the continuous embeddings into reflexive $L^p$ spaces stated in~\eqref{eq_embed}, we can pass to a subsequence with
\begin{align*}
	f_{n_k} \rightharpoonup f \text{ weakly in } L^4(\R^3) \cap L^\frac{4}{3}(\R^3),
	\quad
	u_{n_k} \rightharpoonup u \text{ weakly in } L^4(\R^3) 
\end{align*}
for some $f \in L_\text{rad}^4(\R^3) \cap L_\text{rad}^\frac{4}{3}(\R^3)$, $u \in L^4_\text{rad}(\R^3)$. 
Using Proposition~A.1 in~\cite{EvequozWeth} and a suitable diagonal sequence, we may assume that the subsequence satisfies $u_{n_k} \rightharpoonup u$ weakly in $W^{2,4}_\text{loc}(\R^3)$ and hence $u_{n_k} \to u$ strongly in $C^1_\text{loc}(\R^3)$, the latter thanks to the Rellich-Kondrachov Embedding Theorem~6.3 in\cite{Adams}. 
We will prove that $(u_{n_k})_k$ is a Cauchy sequence in $X_1$. 
We let $\varepsilon > 0$ and choose 
\begin{align}\label{eq_R(eps)}
	R := \max\left\{ \frac{8 \sqrt{2}}{\varepsilon \sqrt{\lambda}} \cdot \sup_{n\in\N} \norm{f_n}_{X_3}, 1\right\}.
\end{align}
Then since $u_{n_k} \to u$ in $C^1_\text{loc}(\R^3)$, in particular, $u_{n_k} \to u$ uniformly on $\bar B_R(0)$, and we can choose $k_1(\varepsilon) \in \N$ with 
\begin{align}\label{eq_R(eps)-4}
	\sup_{|x| \leq R} (1 + |x|^2)^\frac{1}{2} |u_{n_k}(x) - u_{n_l}(x)| < \varepsilon
	\qquad
	\text{for all } k, l \geq k_1(\varepsilon).
\end{align}
We next observe that Lemma~\ref{lem_fourier} implies $\hat{f}_{n_k}(\sqrt{\lambda}) \to \hat{f}(\sqrt{\lambda})$ as $k \to \infty$ since $f_{n_k} \rightharpoonup f$ weakly in $L^\frac{4}{3}(\R^3)$. Thus, we find $k_2(\varepsilon) \in \N$ with the property that
\begin{align}\label{eq_R(eps)-2}
	\sqrt{\pi} |\hat{f}_{n_k}(\sqrt{\lambda}) - \hat{f}_{n_l}(\sqrt{\lambda})|  < \frac{\varepsilon}{2}
	\qquad \text{for all } k, l \geq k_2(\varepsilon).
\end{align}
Since $f_n \in X_3$, we estimate for $|x| > R \geq 1$ and $k, l \geq k_2(\varepsilon)$ using~(d)
\begin{align*}
	&(1 + |x|^2)^\frac{1}{2}|u_{n_k}(x) - u_{n_l}(x)|
	\\
	&\quad \overset{\eqref{eq_R(eps)-3}}{\leq} 
	\sqrt{\frac{\pi}{2}} |\hat{f}_{n_k}(\sqrt{\lambda}) - \hat{f}_{n_l}(\sqrt{\lambda})|  \frac{(1 + |x|^2)^\frac{1}{2}}{|x|} 
	+ \frac{2(1 + |x|^2)^\frac{1}{2}}{\sqrt{\lambda} |x|^2} \cdot \norm{f_{n_k} - f_{n_l}}_{X_3}
	\\
	& \quad \leq \sqrt{\pi} |\hat{f}_{n_k}(\sqrt{\lambda}) - \hat{f}_{n_l}(\sqrt{\lambda})| 
	+ \frac{2\sqrt{2}}{\sqrt{\lambda} R} \cdot 2 \sup_{n \in \N} \norm{f_n}_{X_3}
	\\
	& \overset{\eqref{eq_R(eps)}, \eqref{eq_R(eps)-2}}{<}  \varepsilon.
\end{align*}
Combining this with~\eqref{eq_R(eps)-4}, we have 
\begin{align*}
	\norm{u_{n_k} - u_{n_l}}_{X_1} = \sup_{|x| \geq 0} (1 + |x|^2)^\frac{1}{2}|u_{n_k}(x) - u_{n_l}(x)| < \varepsilon
	\qquad \text{for all } k, l \geq \max\{ k_1(\varepsilon), k_2(\varepsilon)\}.
\end{align*}
Hence $(u_{n_k})_{k \in \N}$ is a Cauchy sequence in $X_1$, which implies $u_{n_k} \to u$ strongly in $X_1$.
\hfill$\square$

\subsection{Proof of Corollary~\ref{cor_asymptotic}}

Let $f \in X_3$ and $\omega \in (0, \pi)$.
First, given $w \in X_1$ with $w = \mathcal{R}_\lambda^\omega f$,  Proposition~\ref{prop_cont-comp}~(c) implies that $w \in C^2(\R^3)$ and $(- \Delta - \lambda) \: w  = f$ on $\R^3$. Proposition~\ref{prop_cont-comp}~(d) with $\alpha = 1, \tilde{\alpha} = \cot (\omega)$ further states
\begin{align*}
	w(x) &= \sqrt{\frac{\pi}{2}} \hat{f}(\sqrt{\lambda}) \cdot 
	\frac{\sin(\omega)\cos(|x|\sqrt{\lambda}) + \cos(\omega)\sin(|x|\sqrt{\lambda})}{\sin(\omega) |x|} + O\left(\frac{1}{|x|^2}\right)
	\\
	&= \sqrt{\frac{\pi}{2}} \frac{\hat{f}(\sqrt{\lambda})}{\sin(\omega)} \cdot 
	\frac{\sin(|x|\sqrt{\lambda} + \omega)}{|x|} + O\left(\frac{1}{|x|^2}\right)
\end{align*}
as $|x| \to \infty$, as asserted. Conversely, let us now assume that $w \in X_1$ solves $- \Delta w - \lambda w = f$ on $\R^3$ and satisfies
\begin{align*}
	w(x) = \gamma \cdot \frac{\sin(|x|\sqrt{\lambda} + \omega)}{|x|} + O\left(\frac{1}{|x|^2}\right)
\end{align*}
as $|x| \to \infty$ for some $\gamma \in \R$. Then $w - \mathcal{R}_\lambda^\omega f$ is a radially symmetric solution of the homogeneous Helmholtz equation, hence $w - \mathcal{R}_\lambda^\omega f = \tilde{\alpha} \cdot \tilde{\Psi}_\lambda$ for some $\tilde{\alpha} \in \R$.
Due to Proposition~\ref{prop_cont-comp}~(d), 
\begin{align*}
	\tilde{\alpha} \cdot \tilde{\Psi}_\lambda(x)
	= w(x) - \mathcal{R}_\lambda^\omega f (x)
	= \left( \gamma - \sqrt{\frac{\pi}{2}} \hat{f}(\sqrt{\lambda}) \right) 
	\cdot \frac{\sin(|x|\sqrt{\lambda} + \omega)}{|x|} + O\left(\frac{1}{|x|^2}\right)
\end{align*}
as $|x| \to \infty$, which implies $\tilde{\alpha} = 0$ and hence $w =\mathcal{R}_\lambda^\omega f$.
\hfill$\square$

\subsection{Proof of Corollary~\ref{cor_alpha-beta}}
\begin{itemize}
\item[(a)] 
Assuming that $w \in X_1$ is a $C^2$ solution of $- \Delta w - \lambda w = f$ on $\R^3$, Proposition~\ref{prop_cont-comp}~(c) implies that $w - \Psi_\lambda \ast f$ is a radially symmetric solution of the homogeneous equation $- \Delta v - \lambda v = 0$. Hence $w - \Psi_\lambda \ast f = \tilde{\alpha} \cdot \tilde{\Psi}_\lambda$, $w \in U_1(\lambda)$, and $\tilde{\alpha} = \alpha^{(\nu)}(w)$ by definition of the functional $\alpha^{(\nu)}$. Conversely, assuming that $w = \Psi_\lambda \ast f + \alpha^{(\nu)}(w) \cdot \tilde{\Psi}_\lambda$, the assertion follows from Proposition~\ref{prop_cont-comp}~(c).
\item[(b)]
If we assume that $w \in X_1$ is a $C^2$ solution of 
\begin{align*}
	- \Delta w - \lambda w = f \quad \text{on } \R^3, 
	\qquad
	w(x) = \gamma \: \frac{\sin(|x|\sqrt{\lambda})}{|x|} + O\left(\frac{1}{|x|^2}\right) \quad \text{as } |x| \to \infty,
\end{align*}
we infer $w = \Psi_\lambda \ast f + \alpha^{(\nu)}(w) \cdot \tilde{\Psi}_\lambda$ from (a) and $\beta^{(\nu)}(w) = 0$ from the defining equations~\eqref{eq_alpha-beta} since $w \in U_1(\lambda)$. On the other hand, given that $w = \Psi_\lambda \ast f + (\alpha^{(\nu)}(w) + \sigma \beta^{(\nu)}(w)) \cdot \tilde{\Psi}_\lambda$, we deduce that $w$ is twice continuously differentiable with $- \Delta w - \lambda w = f$ from Proposition~\ref{prop_cont-comp}~(c). By assumption, $w \in U_1(\lambda)$, and applying $\alpha^{(\nu)}$ to the identity $w = \Psi_\lambda \ast f + (\alpha^{(\nu)}(w) + \sigma \beta^{(\nu)}(w)) \cdot \tilde{\Psi}_\lambda$ yields $\beta^{(\nu)}(w) = 0$ thanks to the equations~\eqref{eq_alpha-beta-coeff}. 
\end{itemize}
\hfill $\square$

\subsection{Proof of Proposition~\ref{prop_pruefer}}

Let $g \in X_2$. Then the profile $w: [0, \infty) \to \R$ is a (global) solution of the initial value problem~\eqref{eq_single-pruefer} if and only if $y: [0, \infty) \to \R, y(r) = r \cdot w(r)$ solves
\begin{align}\label{eq_pertHO}
	\begin{cases}
		-y'' - \lambda y = g(r) \cdot y	& \text{on } (0, \infty),
		\\
		y(0) = 0, y'(0) = 1.
	\end{cases}
\end{align}
Moreover, $w \in X_1$ if $y$ is bounded. Global existence and uniqueness of such $y \in C^2([0, \infty))$ are consequences of the Picard-Lindelöf Theorem and of Gronwall's Lemma since $g \in L^1([0, \infty))$. 
Our proof of boundedness of $y$ and of the asserted asymptotic expansions is inspired by perturbation results due to Hartman in~\cite{hartman}. It is an application of the Prüfer transformation, see equation~(2.1)~in~\cite{hartman}. Since $y \not\equiv 0$, uniqueness implies that $y(r)^2 + y'(r)^2 > 0$ for all $r \geq 0$. We thus parametrize using polar coordinates in the phase space
\begin{align}\label{eq_pruefer-para}
	y(r) = \varrho(r) \cdot \sin(\phi(r) \sqrt{\lambda}),
	\qquad
	y'(r) = \varrho(r) \cdot \sqrt{\lambda} \cos(\phi(r) \sqrt{\lambda})
	\qquad
	(r \geq 0)
\end{align}
with functions $\varrho: [0, \infty) \to (0, \infty)$ and $\phi: [0, \infty) \to \R$. A short calculation shows that we thus obtain a solution of~\eqref{eq_pertHO} if and only if $\varrho$ and $\phi$ satisfy the first-order system
\begin{align}\label{eq_pruefer}
	\begin{cases}
		(\log \varrho)' = - \frac{g(r)}{2\sqrt{\lambda}} \sin(2 \phi \sqrt{\lambda})	& \text{on } (0, \infty),
		\\
		\phi' = 1 + \frac{g(r)}{\lambda} \sin^2(\phi \sqrt{\lambda})					& \text{on } (0, \infty),
		\\
		\varrho(0) = \frac{1}{\sqrt{\lambda}}, \quad \phi(0) = 0.
	\end{cases}
\end{align}
Eqivalently, for $r \geq 0$,
\begin{equation}\label{eq_pruefer-solutions}
\begin{split}
	\varrho (r) 
	&= \frac{1}{\sqrt{\lambda}} \cdot 
	\exp\left(- \int_0^r \frac{g(t)}{2\sqrt{\lambda}} \sin(2 \phi(t) \sqrt{\lambda}) \: \mathrm{d}t\right),
	\\
	\phi (r)
	&= r + \int_0^r \frac{g(t)}{\lambda} \sin^2(\phi(t) \sqrt{\lambda}) \: \mathrm{d}t.
\end{split}
\end{equation}
We will frequently refer to the estimate
\begin{align}\label{eq_pruefer-estimate-rho}
	\forall \: r \geq 0 \qquad
	\frac{1}{\sqrt{\lambda}} 
	\exp\left(\left| \int_0^r \frac{g(t)}{2\sqrt{\lambda}} \sin(2 \phi(t) \sqrt{\lambda}) \: \mathrm{d}t\right|\right)
	\leq \frac{1}{\sqrt{\lambda}} \exp\left(\frac{\pi}{4\sqrt{\lambda}} \:  \norm{g}_{X_2}\right) =: C_g.
\end{align}
Indeed, it immediately yields boundedness of the solution since, due to $g \in X_2$,
\begin{align*}
	|y(r)| \overset{\eqref{eq_pruefer-para}}{\leq} |\varrho(r)| 
	&\overset{\eqref{eq_pruefer-solutions}}{\leq}
	\frac{1}{\sqrt{\lambda}} \cdot 
	\exp\left(\left| \int_0^r \frac{g(t)}{2\sqrt{\lambda}} \sin(2 \phi(t) \sqrt{\lambda})	\: \mathrm{d}t \right|  \right) 
	\overset{\eqref{eq_pruefer-estimate-rho}}{\leq} C_g.
\end{align*}
Analogously, we see that the improper integrals in
\begin{align*}
	\omega_\lambda(g) := \int_0^\infty \frac{g(t)}{\sqrt{\lambda}} \sin^2(\phi(t) \sqrt{\lambda}) \: \mathrm{d}t
	\quad \text{and} \quad
	\varrho_\lambda(g) := \frac{1}{\sqrt{\lambda}} \cdot 
	\exp\left(- \int_0^\infty \frac{g(t)}{2\sqrt{\lambda}} \sin(2 \phi(t) \sqrt{\lambda}) \: \mathrm{d}t\right)
\end{align*}
converge, observe $\varrho_\lambda (g) > 0$, and verify the asserted asymptotic behavior of $y$ as $r \to \infty$:
\begin{align*}
	&\left| y(r) - \varrho_\lambda(g) \sin(r \sqrt{\lambda} + \omega_\lambda(g)) \right|
	\\
	& \quad \overset{\eqref{eq_pruefer-para}}{=} \left| \varrho(r) \sin(\phi(r) \sqrt{\lambda})
	- \varrho_\lambda(g) \sin(r \sqrt{\lambda} + \omega_\lambda(g)) \right|
	\\
	& \quad \leq
	\left| \left[ \varrho(r)  - \varrho_\lambda(g)\right] \sin(\phi(r) \sqrt{\lambda}) \right|
	+ \left| \varrho_\lambda(g)  \left[\sin(\phi(r) \sqrt{\lambda}) 
	- \sin(r \sqrt{\lambda} + \omega_\lambda(g)) \right] \right|
	\\
	& \quad \overset{\eqref{eq_pruefer-estimate-rho}}{\leq}
	C_g \cdot \left(  
	\left| \frac{\varrho(r)}{\varrho_\lambda(g)} - 1 \right|
	+ 	 \left|\sin(\phi(r) \sqrt{\lambda}) - \sin(r \sqrt{\lambda} + \omega_\lambda(g)) \right| \right)
\end{align*}
where we estimate both terms as follows
\begin{align*}
	\left| \frac{\varrho(r)}{\varrho_\lambda(g)} - 1 \right|
	& 
	\overset{\eqref{eq_pruefer-solutions}}{=} \left|
	\exp\left(\int_r^\infty \frac{g(t)}{2\sqrt{\lambda}} \sin(2 \phi(t) \sqrt{\lambda}) \: \mathrm{d}t\right) - 1
	\right|
	\\
	&\overset{\eqref{eq_pruefer-estimate-rho}}{\leq} 
	C_g \cdot 
	\left| \int_r^\infty \frac{g(t)}{2\sqrt{\lambda}} \sin(2 \phi(t) \sqrt{\lambda}) \: \mathrm{d}t \right|
	\\
	&\leq
	C_g \cdot \frac{\norm{g}_{X_2}}{2\sqrt{\lambda}} \cdot
	\int_r^\infty \frac{\mathrm{d}t}{1 + t^2}
	\\
	&\leq
	C_g \cdot \frac{\norm{g}_{X_2}}{2\sqrt{\lambda}} \cdot \frac{1}{r},
	\\
	\left|\sin(\phi(r) \sqrt{\lambda}) - \sin(r \sqrt{\lambda} + \omega_\lambda(g)) \right|
	& \leq
	\left|\phi(r) \sqrt{\lambda} - r \sqrt{\lambda} - \omega_\lambda(g) \right|
	\\
	& =
	\left|\int_r^\infty \frac{g(t)}{\sqrt{\lambda}} \sin^2(\phi(t) \sqrt{\lambda}) \: \mathrm{d}t \right|
	\\
	& \leq
	\frac{\norm{g}_{X_2}}{\sqrt{\lambda}} \cdot 
	\int_r^\infty \frac{\mathrm{d}t}{1 + t^2}  
	\\
	& \leq
	\frac{\norm{g}_{X_2}}{\sqrt{\lambda}} \cdot 
	\frac{1}{r}.
\end{align*}
Thus $y(r) - \varrho_\lambda(g) \sin(r \sqrt{\lambda} + \omega_\lambda(g)) = O\left( \frac{1}{r} \right)$ as $r \to \infty$. Similarly one can show that $y'(r) - \varrho_\lambda(g) \sqrt{\lambda} \cos(r \sqrt{\lambda} + \omega_\lambda(g)) = O\left( \frac{1}{r} \right)$ as $r \to \infty$. Since $y(r) = r \cdot w(r)$, the Proposition is proved.
\hfill $\square$

\subsection{Proof of Proposition~\ref{prop_aspt-phase-continuous}}

We consider $g_n, g_0 \in X_2$ with $g_n \to g_0 \in X_2$ and aim to show that $\omega_\lambda(g_n) \to \omega_\lambda(g_0)$.
By $\phi_n \in C^1((0, \infty)) \cap C([0, \infty))$, we denote the unique solution of
\begin{align*}
	\phi_n' = 1 + \frac{g_n(r)}{\lambda} \sin^2(\phi_n \sqrt{\lambda}),
	\qquad
	\phi_n(0) = 0.
\end{align*}
Then we have pointwise convergence, $\phi_n(r) \to \phi_0(r)$ for all $r \geq 0$. Indeed, let us fix any $R > 0$ and estimate for $0 \leq r \leq R$ and $n \in \N$
\begin{align*}
	\left| \phi_n(r) - \phi_0(r) \right|
	& = 
	\left|\int_0^r  \frac{g_n(t)}{\lambda} \sin^2(\phi_n(t) \sqrt{\lambda}) - \frac{g_0(t)}{\lambda} \sin^2(\phi_0(t) \sqrt{\lambda})
	\: \mathrm{d}t \right|
	\\
	& \leq
	\frac{1}{\lambda} \int_0^r  \left| g_n(t) - g_0(t) \right| \: \mathrm{d}t
	+ 	\frac{1}{\lambda} \int_0^r  \left|g_0(t) \right| \cdot
	\left| \sin^2(\phi_n(t) \sqrt{\lambda}) - \sin^2(\phi_0(t) \sqrt{\lambda}) \right| \: \mathrm{d}t
	\\
	& \leq
	\frac{1}{\lambda} \int_0^\infty  \norm{g_n - g_0}_{X_2} \: \frac{\mathrm{d}t}{1+t^2}
	+ 	\frac{2 \norm{g_0}_\infty}{\sqrt{\lambda}} \int_0^r 
	\left| \phi_n(t)  - \phi_0(t)  \right| \: \mathrm{d}t
	\\
	& \leq
	\frac{\pi}{2 \lambda} \norm{g_n - g_0}_{X_2}
	+ 	\frac{2 \norm{g_0}_\infty}{\sqrt{\lambda}} \int_0^r 
	\left| \phi_n(t)  - \phi_0(t)  \right| \: \mathrm{d}t.
\end{align*}
Thus, by Gronwall's Lemma, we have for $0 \leq r \leq R$
\begin{align*}
	\left| \phi_n(r) - \phi_0(r) \right| 
	\leq \frac{\pi}{2 \lambda} \norm{g_n - g_0}_{X_2}
	 \cdot \mathrm{e}^{	\frac{2 \norm{g_0}_\infty}{\sqrt{\lambda}} \: r}
	\leq \frac{\pi}{2 \lambda} \norm{g_n - g_0}_{X_2}
	 \cdot \mathrm{e}^{	\frac{2 \norm{g_0}_\infty}{\sqrt{\lambda}} \: R}.
\end{align*}
Since $g_n \to g_0$ in $X_2$, we conclude $\phi_n \to \phi_0$ locally uniformly on $[0, \infty)$, in particular pointwise.
Now we can deduce the convergence of the asymptotic phase,
\begin{align*}
	\omega_\lambda(g_n) = \frac{1}{\sqrt{\lambda}} \int_0^\infty g_n(r) \sin^2(\phi_n(r)\sqrt{\lambda}) \, \mathrm{d}r
	\to \frac{1}{\sqrt{\lambda}} \int_0^\infty g_0(r) \sin^2(\phi_0(r)\sqrt{\lambda}) \, \mathrm{d}r = \omega_\lambda(g_0),
\end{align*}
which is a consequence of the Dominated Convergence Theorem.
Indeed, the integrands converge pointwise and are integrably majorized by
\begin{align*}
	\left| g_n(r) \sin^2(\phi_n(r)\sqrt{\lambda}) \right|
	\leq \frac{\sup_{n \in \N} \norm{g_n}_{X_2}}{1 + r^2} 
	\qquad
	(r \geq 0).
\end{align*}
\hfill $\square$

\subsection{Proof of Proposition~\ref{prop_asymptoticphase}}

Let us first recall that, given the assumptions of Proposition~\ref{prop_asymptoticphase}, equation~\eqref{eq_aspt-phase} implies for $b \in \R$
\begin{align*}
	\omega_\lambda(b \, u_0^2) = \frac{b}{\sqrt{\lambda}} \int_0^\infty u_0^2(r) \sin^2(\phi_b(r) \sqrt{\lambda}) \: \mathrm{d}r
\end{align*}
where $\phi_b$ satisfies $\phi_b' = 1 + \frac{b}{\lambda} u_0^2(r) \sin^2(\phi_b \sqrt{\lambda})$ on $(0, \infty)$, $\phi_b(0) = 0$.
We immediately see that $\omega_\lambda(0) = 0$ and $\text{sgn } \omega_\lambda(b \, u_0^2) = \text{sgn }(b)$ for all $b \in \R \setminus \{ 0 \}$. Further, continuity of $b \mapsto \omega_\lambda(b \, u_0^2)$ is a consequence of Proposition~\ref{prop_aspt-phase-continuous}. The assertions are proved once we show that $b \mapsto \omega_\lambda(b \, u_0^2)$ is strictly increasing with
\begin{align*}
	\omega_\lambda(b \, u_0^2) \to \pm \infty
	\qquad
	\text{as } b \to \pm \infty.
\end{align*}

\step{1}{Strict monotonicity.}
We let $b_1 < b_2$, define 
\begin{align*}
	\chi(r) := \begin{cases}
	\frac{\sin^2(\phi_{b_2}(r)\sqrt{\lambda}) - \sin^2(\phi_{b_1}(r)\sqrt{\lambda})}{\phi_{b_2}(r)\sqrt{\lambda} - \phi_{b_1}(r)\sqrt{\lambda}}
	& \text{if } \phi_{b_2}(r) \neq \phi_{b_1}(r),
	\\
	2 \sin(\phi_{b_1}(r)\sqrt{\lambda}) \cos(\phi_{b_1}(r)\sqrt{\lambda}) & \text{else}
	\end{cases}
\end{align*}
and observe that $\chi$ is bounded with $0 \leq |\chi(r)| \leq 2$ and continuous. $\psi := \phi_{b_2} - \phi_{b_1}$ satisfies
\begin{align*}
	\psi' = 
	\frac{b_2 - b_1}{\lambda} u_0^2(r) \sin^2(\phi_{b_2}(r) \sqrt{\lambda}) +  \frac{b_1}{\sqrt{\lambda}} u_0^2(r) \chi(r) \psi,
	\qquad
	\psi(0) = 0.
\end{align*}
The unique solution is given by the Variation of Constants formula. We have
\begin{align*}
	&\omega_\lambda(b_2 \, u_0^2) - \omega_\lambda(b_1 \, u_0^2) 
	= \lim_{r \to \infty} \psi(r)
	 = \int_0^\infty \frac{b_2 - b_1}{\lambda} u_0^2(\varrho) \sin^2(\phi_{b_2}(\varrho) \sqrt{\lambda})
	\mathrm{e}^{\int_\varrho^\infty \frac{b_1}{\sqrt{\lambda}} u_0^2(\tau) \chi(\tau) \: \mathrm{d}\tau} \: \mathrm{d}\varrho > 0
\end{align*}
since the integrand is nonnegative and not identically zero.

\step{2}{Asymptotic behavior as $b \to \infty$.}

By the uniqueness statement of the Picard-Lindelöf Theorem, $u_0 \not\equiv 0$ requires $u_0(0) \neq 0$. 
We can thus choose $r_0 > 0$ with 
\begin{align}\label{eq_proofAP-r0}
	\frac{1}{2} u_0^2(0) < u_0^2(r) < \frac{3}{2} u_0^2(0)
	\qquad 
	\text{for all } r \in [0, r_0].
\end{align}
We intend to use a comparison technique. By definition of $r_0$, we have for $b > 0$
\begin{align*}
	\phi_b' = 1 + \frac{b}{\lambda} u_0^2(r) \sin^2(\phi_b \sqrt{\lambda}) 
	\overset{\eqref{eq_proofAP-r0}}{\geq} 1 + \frac{b}{2 \lambda} u_0^2(0) \sin^2(\phi_b \sqrt{\lambda}) 
	\quad \text{ on } [0, r_0],
	\qquad
	\phi_b(0) = 0.
\end{align*}
We now study the modified initial value problem
\begin{align*}
	\psi_b' = 1 + \frac{b}{2 \lambda} u_0^2(0) \sin^2(\psi_b \sqrt{\lambda})
	\quad \text{ on } [0, r_0],
	\qquad
	\psi_b(0) = 0.
\end{align*}
For $0 \leq r \leq r_0$ with $r \not\in \frac{\pi}{2} + \pi \Z$, its unique solution is given by the expression
\begin{equation}\label{eq_proofODE-psib-formula}
\begin{split}
	&\psi_b(r) =
	\frac{1}{\sqrt{\lambda}} \left[ 
	n \pi +   \arctan \left( \frac{\tan \left( r \sqrt{\lambda} \: \sqrt{1 +\frac{b}{2 \lambda} u_0^2(0)} \right)}
	{\sqrt{1 + \frac{b}{2 \lambda} u_0^2(0)}} \right)
	\right]
	\\
	&\text{if } n \in \N_0, \: \: \left|\sqrt{1 + \frac{b}{2 \lambda} u_0^2(0)} \sqrt{\lambda} \: r - n \pi \right| < \frac{\pi}{2}.
\end{split}
\end{equation}
Then $\phi_b \geq \psi_b$ on $[0, r_0]$, in particular $\phi_b(r_0) \geq \psi_b(r_0)$.
Substituting $t := \sqrt{\lambda} \phi_b(r)$, we estimate
\begin{align*}
	\omega_\lambda(b \, u_0^2)
	&= \frac{b}{\sqrt{\lambda}}  \int_0^\infty u_0^2(r) \sin^2(\phi_b(r) \sqrt{\lambda}) \: \mathrm{d}r
	\\
	&\overset{\eqref{eq_proofAP-r0}}{\geq} \frac{b u_0^2(0)}{2 \sqrt{\lambda}} \int_0^{r_0} \sin^2(\phi_b(r) \sqrt{\lambda}) \: \mathrm{d}r
	\\
	&= \frac{b u_0^2(0)}{2 \sqrt{\lambda}} 
	\int_0^{\sqrt{\lambda} \phi_b(r_0)} \sin^2(t) \: \frac{\mathrm{d}t}{\lambda^\frac{1}{2} \phi_b'(\phi_b^{-1}(\lambda^{-\frac{1}{2}}t))}
	\\
	&\overset{\eqref{eq_proofAP-r0}}{\geq} \frac{b u_0^2(0)}{2 \lambda} 
	\int_0^{\sqrt{\lambda} \phi_b(r_0)} \sin^2(t) \: \frac{\mathrm{d}t}{1 + \frac{3}{2} \frac{b}{\lambda} u_0^2(0)}
	\\
	&= \frac{b u_0^2(0)}{2 \lambda + 3 b u_0^2(0)} 
	\int_0^{\sqrt{\lambda} \phi_b(r_0)} \sin^2(t) \: \mathrm{d}t
	\\
	&\geq \frac{b u_0^2(0)}{2 \lambda + 3 b u_0^2(0)} 
	\int_0^{\sqrt{\lambda} \psi_b(r_0)} \sin^2(t) \: \mathrm{d}t
\end{align*}
and hence $\omega_\lambda(b \, u_0^2) \to \infty$ as $b \to \infty$ since formula~\eqref{eq_proofODE-psib-formula} implies $\psi_b(r_0) \to \infty$ as $b \to \infty$.

\step{3}{Asymptotic behavior as $b \to - \infty$.}

For $b < -1$, we introduce 
\begin{align*}
	r_b := \max \left\{ r > 0 \: \bigg| \: \phi_b(r) = \frac{1}{\sqrt{\lambda}} \arcsin (|b|^{-\frac{1}{4}}) \right\}.
\end{align*}
Then $r_b \in (0, \infty)$ is well-defined since, due to $1 - \frac{|b|}{\lambda} \frac{\norm{u_0}^2_{X_1}}{1 + r^2} \leq \phi_b' \leq 1$ and $\phi_b(0) = 0$, we have $r - \frac{|b|}{\lambda} \norm{u_0}^2_{X_1} \arctan(r) \leq \phi_b(r) \leq r$ for $r \geq 0$. In particular, we have 
\begin{align}\label{eq_proofRB}
	\phi_b(r_b) = \frac{1}{\sqrt{\lambda}} \arcsin (|b|^{-\frac{1}{4}})
	\quad \text{and} \quad
	\phi_b(r) > \frac{1}{\sqrt{\lambda}} \arcsin (|b|^{-\frac{1}{4}}) \: \text{for all } r > r_b.
\end{align}

We prove below that $r_b \to \infty$ as $b \to - \infty$. 
Then for $r \geq r_b$, equation~\eqref{eq_proofRB} and $\phi_b' \leq 1$ imply
\begin{align*}
	\phi_b(r) \leq  \phi_b(r_b) + (r - r_b)
	= r + \left( \frac{1}{\sqrt{\lambda}} \arcsin (|b|^{-\frac{1}{4}}) - r_b\right).
\end{align*}
Then the asymptotic phase satisfies
\begin{align*}
	\omega_\lambda(b \, u_0^2) = \sqrt{\lambda} \cdot \lim_{r \to \infty} (\phi_b(r) - r)
	\leq \arcsin (|b|^{-\frac{1}{4}}) - r_b \sqrt{\lambda} \to - \infty
	\quad
	\text{as } b \to -\infty.
\end{align*}
It remains to prove that $r_b \to \infty$ as $b \to - \infty$. 
We assume by contradiction that we find a subsequence $(b_k)_{k \in \N}$ and $\tilde r > 0$ with
$	b_k \searrow - \infty $, $	r_{b_k} \to \tilde r $ as $k \to \infty$.
Then, since $\phi_{b_k}' \leq 1$ and due to equation~\eqref{eq_proofRB}, we have 
\begin{align*}
	\frac{1}{\sqrt{\lambda}} \arcsin (|b_k|^{-\frac{1}{4}}) \leq \phi_{b_k}(r) 
	\leq \frac{1}{\sqrt{\lambda}} \arcsin (|b_k|^{-\frac{1}{4}}) + \frac{1}{\sqrt{\lambda}}
	\qquad
	\text{for } r_{b_k} \leq r \leq r_{b_k} +  \frac{1}{\sqrt{\lambda}}, \: k \in \N.
\end{align*}
Since  $|b_k| \to \infty$ we may assume w.l.o.g. $\frac{1}{\sqrt{\lambda}} \arcsin (|b_k|^{-\frac{1}{4}}) \leq \phi_{b_k}(r) 
	\leq \frac{1}{\sqrt{\lambda}} \cdot \frac{\pi}{2}$, hence
\begin{align}\label{eq_proofRB-est}
	\sin(\phi_{b_k}(r) \sqrt{\lambda}) \geq |b_k|^{-\frac{1}{4}}
	\qquad
	\text{for } r_{b_k} \leq r \leq r_{b_k} +  \frac{1}{\sqrt{\lambda}}, \: k \in \N.
\end{align}
We conclude, as $k \to \infty$, 
\begin{align*}
	\phi_{b_k}\left(r_{b_k} + \frac{1}{\sqrt{\lambda}} \right) 
	&= \phi_{b_k}(r_{b_k}) + \int_0^{\frac{1}{\sqrt{\lambda}}}  \phi_{b_k}'(r_{b_k} + \tau) \: \mathrm{d}\tau 
	\\
	&\overset{\eqref{eq_proofRB}}{=} \frac{1}{\sqrt{\lambda}} \arcsin (|b_k|^{-\frac{1}{4}})
	+ \int_0^{\frac{1}{\sqrt{\lambda}}} \left[ 1 - \frac{|b_k|}{\lambda} u_0^2(r_{b_k} + \tau) 
	\sin^2(\phi_{b_k}(r_{b_k} + \tau) \sqrt{\lambda})  \right]
	\: \mathrm{d}\tau 
	\\
	&\overset{\eqref{eq_proofRB-est}}{\leq} \frac{2}{\sqrt{\lambda}} + \frac{1}{\sqrt{\lambda}} 
	-   \frac{\sqrt{|b_k|}}{\lambda} \cdot \int_0^{\frac{1}{\sqrt{\lambda}}} u_0^2(r_{b_k} + \tau) \: \mathrm{d}\tau 
	\\
	&= \frac{3}{\sqrt{\lambda}} 
	- \frac{\sqrt{|b_k|}}{\lambda} \cdot \left( \int_0^{\frac{1}{\sqrt{\lambda}}} u_0^2(\tilde{r} + \tau) \: \mathrm{d}\tau  + o(1) \right)
	\\
	& \to - \infty
\end{align*}
since $u_0^2 > 0$ almost everywhere. 
On the other hand that, for every $k \in \N$, the differential equation $\phi' = 1 + \frac{b_k}{\lambda} u_0^2(r) \sin^2(\phi \sqrt{\lambda})$ states that $\phi_{b_k}(r) = 0$ implies $\phi_{b_k}'(r) = 1$ and thus $\phi_{b_k}$ cannot attain negative values, which contradicts the limit calculated before.
\hfill $\square$

\subsection{Proof of Proposition~\ref{prop_spectrum}}

For $\omega \in (0, \pi)$ and $\lambda > 0$, we compute the spectrum of the linear operator
\begin{align*}
	\mathbf{R}_\lambda^\omega: X_1 \to X_1, \quad w \mapsto \mathcal{R}_\lambda^\omega [u_0^2 \, w]
	= \left( \Psi_\lambda + \cot(\omega) \tilde{\Psi}_\lambda \right) \ast [u_0^2 \, w].
\end{align*}
Compactness of $\mathbf{R}_\lambda^\omega$ is a consequence of Proposition~\ref{prop_cont-comp}~(b). Then immediately 
$\sigma(\mathbf{R}_\lambda^\omega) = \{ 0 \} \cup \sigma_p(\mathbf{R}_\lambda^\omega)$ with discrete eigenvalues of finite multiplicity. 

\step{1}{Eigenvalues.}

We find the eigenfunctions of $\mathbf{R}_\lambda^\omega$, that is, we look for such $\eta \in \R, \eta \neq 0$ and $w \in X_1$ that 
\begin{align*}
	\mathbf{R}_\lambda^\omega w = \eta \cdot w.
\end{align*}
Corollary~\ref{cor_asymptotic} implies that this is equivalent to $\eta \in \R, \eta \neq 0$ and $w \in X_1 \cap C^2(\R^3)$,
\begin{align*}
	&- \Delta w - \lambda w = \frac{1}{\eta} \cdot u_0^2(x) \: w \: \: \text{on } \R^3
	\\
	&\text{with } w(x) = c_w \frac{\sin(|x|\sqrt{\lambda} + \omega)}{|x|} + O\left( \frac{1}{|x|^2} \right)
	\text{ as } |x| \to \infty
\end{align*}
for some $c_w \in \R$.
By Proposition~\ref{prop_pruefer}, such an eigenfunction exists if and only if 
\begin{align*}
	\omega_\lambda\left(\frac{1}{\eta} \, u_0^2 \right) = \omega + k \pi
	\quad
	\text{for some } k \in \Z;
\end{align*}
in this case, $c_w \neq 0$ and every eigenspace is one-dimensional since the radially symmetric solution $w$ is unique up to multiplication by a constant.
Since we have seen in Proposition~\ref{prop_asymptoticphase} that $\R \to \R, \: b \mapsto \omega_\lambda(b \, u_0^2)$ is strictly increasing and onto, we can define $b_k(\omega, \lambda, u_0^2)$ via $\omega_\lambda(b_k(\omega, \lambda, u_0^2)\, u_0^2) = \omega + k \pi$ for all $k \in \Z$, and conclude
\begin{align*}
	\sigma_p(\mathbf{R}_\lambda^\omega) = \left\{ \frac{1}{b_k(\omega, \lambda, u_0^2)} \bigg| k \in \Z \right\}.
\end{align*}

\step{2}{Simplicity.}

It remains to show that the eigenvalues are algebraically simple. We consider an eigenvalue $\eta := \frac{1}{b_k(\omega, \lambda, u_0^2)}$ of $\mathbf{R}_\lambda^\omega$ with eigenspace $\ker \left( \mathbf{R}_\lambda^\omega - \eta I_{X_1} \right) = \text{span } \{ w \}$. We have to prove that
\begin{align*}
	\ker \left( \mathbf{R}_\lambda^\omega - \eta I_{X_1} \right)^2 = 
	\ker \left( \mathbf{R}_\lambda^\omega - \eta I_{X_1} \right).
\end{align*}
So let now $v \in \ker \left( \mathbf{R}_\lambda^\omega - \eta I_{X_1} \right)^2$. We assume for contradiction that $v \not\in \ker\left( \mathbf{R}_\lambda^\omega - \eta I_{X_1} \right)$. 

By assumption on $v$, we have $\mathbf{R}_\lambda^\omega v - \eta v \in \ker \left( \mathbf{R}_\lambda^\omega - \eta I_{X_1} \right) \setminus \{ 0 \}$, and since $\eta \neq 0$ we may assume without loss of generality
\begin{align*}
	\mathbf{R}_\lambda^\omega v - \eta v = - \eta w = - \mathbf{R}_\lambda^\omega w.
\end{align*}
We observe that $v, w \in C^2(\R^3)$ by Proposition~\ref{prop_cont-comp} as well as
\begin{align}\label{eq_proof-yz}
	- w'' - \frac{2}{r} w' - \lambda w = \frac{1}{\eta} \:  u_0^2(r) \cdot w,
	\quad
	- v'' - \frac{2}{r} v' - \lambda v = \frac{1}{\eta} \:  u_0^2(r) \cdot (v + w)
	\qquad \text{on } (0, \infty).
\end{align}
Furthermore, Proposition~\ref{prop_cont-comp} (d) implies
\begin{equation}\label{eq_proof-yz-asymptotic}
\begin{split}
	&w(r) = c_w \cdot \frac{\sin(r \sqrt{\lambda} + \omega)}{r} + O\left(\frac{1}{r^2}\right),
	\quad
	w'(r) = c_w \sqrt{\lambda} \cdot \frac{\cos(r \sqrt{\lambda} + \omega)}{r} + O\left(\frac{1}{r^2}\right),
	\\
	&v(r) = c_v \cdot \frac{\sin(r \sqrt{\lambda} + \omega)}{r} + O\left(\frac{1}{r^2}\right),
	\quad
	v'(r) = c_v \sqrt{\lambda} \cdot \frac{\cos(r \sqrt{\lambda} + \omega)}{r} + O\left(\frac{1}{r^2}\right)
\end{split}
\end{equation}
for some $c_w, c_v \in \R$.
Let us define $q(r) = r^2 (w'(r) v(r) - v'(r) w(r))$ for $r \geq 0$. Then, using the differential equations~\eqref{eq_proof-yz}, we find
$q'(r) = \frac{1}{\eta} \: r^2 u_0^2(r) \cdot w^2(r)$ for $r \geq 0$. Hence, depending on the sign of $\eta$, $q$ is monotone on $[0, \infty)$ with $q(0) = 0$.  On the other hand, the asymptotic expansions in~\eqref{eq_proof-yz-asymptotic} imply that $q(r) = O\left(\frac{1}{r}\right)$ as $r \to \infty$. We conclude $q(r) = 0$ for all $r \geq 0$.
Since all zeros of $w$ are simple, one can deduce that $v(r) = c_0 \cdot w(r)$ for all $r \geq 0$ and some $c_0 \in \R$, and thus $v \in \ker\left( \mathbf{R}_\lambda^\omega - \eta I_{X_1} \right)$, a contradiction. 
\hfill $\square$

\subsection{Proof of Proposition~\ref{prop_u0-sigma0-tau0}}

The existence of a continuum of radially symmetric solutions $u_0 \in X_1 \cap C^2(\R^3)$ of the nonlinear Helmholtz equation~\eqref{eq_singleHH} has been shown in Theorem~1.2 of~\cite{radial}. Given such a solution $u_0$, the asymptotic expansion is a consequence of Proposition~\ref{prop_pruefer} applied to equation~\eqref{eq_single-pruefer} with $g := u_0^2 \in X_2$, $\lambda := \mu$ and with unique solution $w := \frac{u_0}{u_0(0)}$. Proposition~\ref{prop_pruefer} also provides a unique radially symmetric solution $w_1 \in X_1 \cap C^2(\R^3)$ of $- \Delta w - \mu w = 3 u_0^2(x) \: w$ on $\R^3$, $w(0) = 1$.
As for $u_0$, the asymptotic behavior of $w_1$ is
\begin{align*}
	w_1(x) = \gamma \cdot \frac{\sin(|x|\sqrt{\mu} + \tau_0)}{|x|} + O\left(\frac{1}{|x|^2}\right)
	\quad \text{as } |x| \to \infty
\end{align*}
for some $\gamma \neq 0$ and $\tau_0 \in [0, \pi)$; then, $w_0 := \frac{w_1}{\gamma}$ has the asserted properties.
\hfill $\square$

\section*{Acknowledgements}

The authors gratefully acknowledge financial support by the Deutsche Forschungsgemeinschaft (DFG) through CRC 1173 ''Wave phenomena: analysis and numerics''. 

\bibliographystyle{abbrv}	
\bibliography{Literatur_Bi} 
  
\end{document}